\documentclass{article}
\usepackage{graphicx}%
\usepackage{multirow}%
\usepackage{amsmath,amssymb,amsfonts}%
\usepackage{amsthm}%
\usepackage{mathrsfs}%
\usepackage[title]{appendix}%
\usepackage{xcolor}%
\usepackage{textcomp}%
\usepackage{manyfoot}%
\usepackage{booktabs}%
\usepackage{algorithm}%
\usepackage{algorithmicx}%
\usepackage{algpseudocode}%
\usepackage{listings}%
\usepackage{color}
\usepackage{enumerate}
\usepackage{hyperref}
\usepackage{times}

\def\RR{{\mathbb R}}

\def\cT{\mathcal{T}}

\newtheorem{theorem}{Theorem}[section]
\newtheorem{lemma}[theorem]{Lemma}
\newtheorem{proposition}[theorem]{Proposition}
\newtheorem{corollary}[theorem]{Corollary}

 \theoremstyle{definition}
\newtheorem{definition}[theorem]{Definition}
\newtheorem{example}[theorem]{Example}

\newtheorem{question}[theorem]{Question}

\theoremstyle{remark}

\numberwithin{equation}{section}

\begin{document}

\title
{Boundary estimates for a fully nonlinear Yamabe problem on Riemannian manifolds}

\author{Weisong Dong\footnote{School of Mathematics, Tianjin University, 135 Yaguan Road, Tianjin, 300354, China.~Email: dr.dong@tju.edu.cn}~, Yanyan Li\footnote{Department of Mathematics, Rutgers University, 110 Frelinghuysen Rd., Piscataway, NJ 08854, USA.~Email: yyli@math.rutgers.edu}~ and Luc Nguyen\footnote{Mathematical Institute and St Edmund Hall, University of Oxford, Andrew Wiles Building, Radcliffe Observatory Quarter, Woodstock Road, 
	    Oxford OX2 6GG, UK.~Email: luc.nguyen@maths.ox.ac.uk}}

\date{}
\maketitle

\begin{abstract}
In this paper, we consider the Dirichlet boundary value problem for fully
nonlinear Yamabe equations on Riemannian manifolds with boundary. Assuming
the existence of a subsolution, we derive \emph{a priori} boundary second derivative estimates and consequently obtain the existence of a smooth solution. Moreover, with respect to a family of equations interpolating the fully nonlinear Yamabe equation and the classical semi-linear Yamabe equation, our estimates remain uniform. Finally, an example of a $C^1$ solution which is smooth in the interior but not smooth at the boundary is also given.


\end{abstract}

\setcounter{tocdepth}{1}
\tableofcontents

\section{Introduction}

Let $(M^n, g)$ be a compact smooth Riemannian manifold of dimension $n \geq 3$ with non-empty smooth boundary $\partial M \neq \emptyset$.
The Schouten tensor of $g$ is defined as
\[
A_g := \frac{1}{n - 2} \Big(\mathrm{Ric}_g - \frac{1}{2 (n - 1)} \mathrm{R}_g  g\Big) ,
\]
where $\mathrm{Ric}_g$ and $\mathrm{R}_g$ are respectively the Ricci tensor and the scalar curvature of
the metric $g$. The Schouten tensor plays an important role in conformal geometry due to its appearance in the Ricci decomposition of the Riemann curvature tensor.

Consider the problem of finding on $M$ a metric $\tilde g$ conformal to $g$ with a prescribed
symmetric function of the eigenvalues of the Schouten tensor $A_{\tilde g}$ and a prescibed Dirichlet boundary data $\tilde g|_{\partial M}$. More precisely, given a function $\psi > 0$ defined on $M$ and a
Riemannian metric $h$ on $\partial M$ which is conformal to $g|_{\partial M}$, the induced metric on $\partial M$,
one looks for a metric $\tilde g$ on $M$ conformal to $g$ such that $\tilde g |_{\partial M} = h$ and the Schouten tensor $A_{\tilde g}$ satisfies
\begin{equation}
\label{conformal-a}
f(\lambda_{\tilde g} (A_{\tilde g}) ) = \psi (x)\; \mbox{in} \; M,
\end{equation}
where $\lambda_{\tilde g} (A_{\tilde g}) = (\lambda_1, \ldots, \lambda_n)$ are the eigenvalues of $A_{\tilde g}$ with respect to $\tilde g$ and the symmetric function $f$ shall be defined later. Note that, under a conformal deformation of metrics $\tilde g = e^{-2u} g$, the Schouten tensor transforms according to 
the formula
\[
A_{\tilde g} = \nabla^2 u + du\otimes du - \frac{1}{2} |\nabla u|_g^2 g + A_g,
\]
where $\nabla$ denotes the covariant derivative of $g$. If $h = e^{- 2\varphi } g|_{\partial M}$ for a function $\varphi \in C^2 (\partial M)$ and $f$ is homogeneous of degree one,
the problem is equivalent to finding a function $u$ on $M$
satisfying  
\begin{equation}
\label{conformal-a'}
\left\{
\begin{aligned}
f(\lambda_{g} (A_{\tilde g})) = &\; \psi (x) e^{-2u}\; & \mbox{in} \; M,\\
u = &\; \varphi \; & \mbox{on} \; \partial M, 
\end{aligned}
\right.
\end{equation}
where $\lambda_{g} (A_{\tilde g})= (\lambda_1, \ldots, \lambda_n)$ are the eigenvalues of $A_{\tilde g}$ with respect to $g$.

The analogous problem on manifolds
without boundary has attracted
much interest,
see e.g. \cite{CGY1, CGY2, CLL-Liouville, FW19, GW, GW2003a, GW2003, GS18, GV2007, HLM, LL03, LL05, LiN1, STW, V2000} and references therein.
On manifolds with boundary, a closely related problem to the above question is to find a conformal metric such that its Schouten tensor
satisfies \eqref{conformal-a} and the mean curvature of the boundary with respect to the new metric is a prescribed function.
This is equivalent to solving \eqref{conformal-a} with a (nonlinear) Neumann boundary condition and there is also a rich literature,
see \cite{ChenSophie, CW23, CLL-bdry1, FT18, JLL, LL06, LiN-2012, LiN2}  and references therein. The above mentioned work are known as fully nonlinear Yamabe problems of positive curvature type. The counterpart for negative curvature type has also been studied intensively  -- see \cite{CHY, DN25, DN25a, DN25b, GLN, Guan08, GSW, GV2003, LiN-2021, LNX, LiSheng, Sui, YWang21, Wu24}.

Let us now give our assumptions on the symmetric function $f$, following Caffarelli, Nirenberg and Spruck \cite{CNS}. 
Let $\Gamma \subset \mathbb{R}^n$ be an open convex symmetric cone with vertex at the origin satisfying
$\Gamma_n \subset \Gamma \subset \Gamma_1$, where $\Gamma_n = \{ \lambda\in \mathbb{R}^n : \lambda_i > 0, \forall 1\leq i \leq n\}$
and $\Gamma_1 = \{\lambda \in \mathbb{R}^n: \sum_{i=1}^n\lambda_i > 0\}$.
Let $f\in C^{\infty } (\Gamma) \cap C^0 (\bar \Gamma)$ be a function which is symmetric in $\lambda_i$
and satisfies
the following structural conditions:
\begin{align}
 &f > 0 \; \mbox{in}\;  \Gamma \; \mbox{and}\; f = 0 \; \mbox{on}\; \partial \Gamma; \tag{f1}
 \label{f1}\\
 &\partial_i f := \frac{\partial f}{\partial \lambda_i} > 0, \; \forall \; 1\leq i\leq n; \tag{f2}
 \label{f2}\\
 &f \; \mbox{is concave in} \;\Gamma; \tag{f3}
 \label{f3}\\
& f \; \mbox{is homogeneous of degree one, i.e. }\;  f(s \lambda) = s f(\lambda), \; \forall\;  s > 0.\tag{f4}
 \label{f4}
\end{align}
We remark that \eqref{f2} ensures that the equation \eqref{conformal-a'} is elliptic and
\eqref{f3} enables us to use Evans-Krylov's Theorem.

Note that, if we define
\[
e = (1, \ldots, 1),
\]
then, by the homogeneity and concavity of $f$,
\begin{equation}
\sum_i \partial_i f (\lambda) = f (\lambda) + D f (\lambda) \cdot (e - \lambda) \geq f (e) > 0 \; \mathrm{ in } \; \Gamma. \tag{f5}
	\label{f5}
\end{equation}
By the symmetry and homogeneity of $f$, we have $D f(e) = \frac{1}{n}f(e) e$. Together with the concavity and once again the homogeneity of $f$, this gives
\begin{equation}
f(\lambda) \leq f\big(\frac{1}{n}\sigma_1(\lambda)e\big) + D f\big(\frac{1}{n}\sigma_1(\lambda)e\big) \cdot (\lambda - \frac{1}{n}\sigma_1(\lambda)e) = \frac{1}{n}\sigma_1(\lambda)f(e).
	\tag{f6}
	\label{f6}
\end{equation}

Classical examples satisfying \eqref{f1}-\eqref{f4} are $\sigma_k^{\frac{1}{k}} (\lambda)$ for $1\leq k \leq n$
and the quotient $\Big( \frac{\sigma_k (\lambda)}{\sigma_\ell (\lambda)}\Big)^{\frac{1}{k-\ell}}$ for $1\leq \ell < k \leq n$
defined on the open convex symmetric cone $\Gamma_k$ (see \cite{CNS}),
where $\sigma_k$ is the $k$-th elementary symmetric function
\[
\sigma_k (\lambda) : = \sum_{1 \leq i_1 <\cdots < i_k \leq n } \lambda_{i_1} \cdots \lambda_{i_k}
\]
and
\[
\Gamma_k : = \{\lambda = (\lambda_1, \ldots, \lambda_n) \in \mathbb{R}^n : \sigma_j (\lambda) > 0, \forall 1 \leq j \leq k\}.
\]
When $f = \sigma_k^{\frac{1}{k}}$ and $\psi = 1$, equation \eqref{conformal-a} is known as the $\sigma_k$-Yamabe problem.

The solvability of \eqref{conformal-a'} has been studied in a number of work, see Schn\"urer \cite{Schnurer}, Guan \cite{Guan07}, Guan and Jiao \cite{GJ}, Li and Nguyen \cite{LiN}. The solvability for related Hessian-type equations has been considered by many authors - we only mention Guan \cite{Guan99, Guan14, Guan18}, Guan, Spruck and Xiao \cite{GSX}, Guan and Zhang \cite{PGZ}, Li \cite{Li90}, Lu \cite{Lu23}, Trudinger \cite{Trudinger}, Wang and Xiao \cite{WX} which are more closely related to the present work and refer the reader to them for further references. In the cited work for \eqref{conformal-a'} as well as in the present work, it is assumed that \eqref{conformal-a'} admits a smooth subsolution. We briefly recall here aspects of \cite{LiN} which are of special relevance. In \cite{LiN}, as an intermediate step in their study of Green’s functions to general nonlinear Yamabe problems, the authors proved that there exists a solution in $C^{0,1} (M) \cap C^\infty (M \setminus \partial M)$ to the Dirichlet problem of \eqref{conformal-a'} (see \cite[Theorem 4.1]{LiN}). The approach in \cite{LiN} was to consider a family of problems generalizing \eqref{conformal-a'}, namely
\begin{equation}
\label{conformal-at}
\left\{
\begin{aligned}
f_t(\lambda_{g} (A_{\tilde g})) = &\; \psi (x) e^{-2u}\; & \mbox{in} \; M,\\
u = &\; \varphi \; & \mbox{on} \; \partial M, 
\end{aligned}
\right.
\end{equation}
where $t \in [0,1]$ and where $f_t$ is defined on the cone
\begin{equation}
\Gamma_t = \big\{\lambda :  t \lambda + (1-t)\sigma_1(\lambda)e \in \Gamma\big\}
	\label{Eq:Gt}
\end{equation}
by
\begin{equation}
f_t(\lambda) = f\big(t \lambda + (1-t)\sigma_1(\lambda)e\big).
	\label{Eq:ft}
\end{equation}
The proof of \cite{LiN} was based on the observation that the proof in \cite{Guan07} can be applied to obtain a solution $u_t \in C^\infty(M)$ of \eqref{conformal-at} for $t \in [0,1)$ and that the family $\{u_t\}_{t \in [0,1)}$ is bounded in $C^1(M) \cap C^\infty_{\rm loc}(M \setminus \partial M)$. Therefore, one obtains  in the limit $t \nearrow 1^-$ a solution of \eqref{conformal-at} for $t = 1$, i.e. \eqref{conformal-a'}, in $C^{0,1} (M) \cap C^\infty (M \setminus \partial M)$. On the other hand, since there are examples of solutions of \eqref{conformal-a'} which belong to $(C^1(M) \cap C^\infty(M \setminus \partial M)) \setminus C^2(M)$ (see Example \ref{Ex1} below), it is natural to ask if the solution of \eqref{conformal-a'} constructed via \eqref{conformal-at} in \cite{LiN} is smooth at $\partial M$ or not. In this work, we prove that it is indeed smooth at $\partial M$. In other words, the solution in \cite[Theorem 4.1]{LiN} belongs to $C^\infty (M)$.

We shall study a Dirichlet problem with a slightly more general right hand side than that in \eqref{conformal-at}. For a function $u\in C^2(M)$, let
\[
W[u] : = A_{e^{-2u}g} = \nabla^2 u + du\otimes du - \frac{1}{2} |\nabla u|^2 g + A_g,
\]
and let $\lambda_g (W[u]) = (\lambda_1, \ldots, \lambda_n)$ denote the eigenvalues of $W[u]$ with respect to the metric $g$. When it is clear from the context, we sometimes write $\lambda(W[u])$ in place of $\lambda_g (W[u])$.  Given $0 < \psi (x, z) \in C^{\infty} (M \times \mathbb{R})$ with $\psi_z \leq 0$ and $\varphi \in C^{\infty} (\partial M)$,
consider the Dirichlet problem
\begin{equation}
\label{eqn}
\left\{
\begin{aligned}
f_t(\lambda_g (W[u])) = &\; \psi (x,u),\; & \mbox{in}\; M,\\
u = &\; \varphi,\;  & \mbox{on}\; \partial M.
\end{aligned}
\right.
\end{equation}
Note that, by \eqref{f3}, \eqref{f4} and \eqref{f6},
\[
f_t(\lambda) \geq t f(\lambda) + (1-t) \sigma_1(\lambda) f(e) \geq (t + n(1-t))f(\lambda) \geq f(\lambda).
\]
We therefore say $\underline{u} \in C^{2} (M)$ is a \emph{subsolution} to \eqref{eqn} for all $t \in [0,1]$ if $\underline{u}$ satisfies that
\[
	\left\{
	\begin{aligned}
		f(\lambda_g (W[\underline{u}])) \geq &\; \psi (x, \underline{u}),\; & \mbox{in}\; M,\\
		\underline{u} = &\; \varphi,\;  & \mbox{on}\; \partial M.
	\end{aligned}
	\right.
\]

The following is our main result.

\begin{theorem}
\label{thm}
Let $(M^n, g)$ be a compact smooth Riemannian manifold of dimension $n \geq 3$ with non-empty smooth boundary $\partial M \neq \emptyset$. Let $\Gamma \subset \mathbb{R}^n$ be an open convex symmetric cone with vertex at the origin satisfying
$\Gamma_n \subset \Gamma \subset \Gamma_1$ and $f\in C^{\infty } (\Gamma) \cap C^0 (\bar \Gamma)$ be a symmetric function
satisfying \eqref{f1}-\eqref{f4}. Let $\psi (x, z) \in C^{\infty} (M \times \mathbb{R})$ be positive with $\psi_z \leq 0$, and suppose that there exists a subsolution $\underline{u} \in C^4 (M)$ to the Dirichlet problem \eqref{eqn} for all $t \in [0,1]$.
Then, for every $t \in [0,1]$ there exists a solution $u_t \in C^\infty (M)$ to the Dirichlet problem \eqref{eqn} satisfying $u_t \geq \underline{u}$.
Moreover,
we have the estimate
\[
\sup_{t \in [0,1]} \sup_{M}  (|u_t| + |\nabla u_t| + |\nabla^2 u_t| )\leq C,
\]
where $C$ only depends on $\underline{u}$, $(M, g)$, $(f, \Gamma)$, $\psi$ and $\varphi$.
\end{theorem}

Our main contribution concerns the uniformity of our estimates with respect to $t$. Available estimates in the literature for a single pair $(f,\Gamma)$ depends on the type of $f$ and/or $\Gamma$ (see Section \ref{Sec2} for terminology). As $t \nearrow 1^-$, the type of $f_t$ and/or $\Gamma_t$ may change, and this calls for care when considering uniform estimates in this limit. Our proof uses a particular property of the family $(f_t,\Gamma_t)$, namely Lemma \ref{Lem:GtBall} in Section \ref{Sec2}. If one replaces the family $(f_t,\Gamma_t)$ by another (suitably continuous) family $(\tilde f_t, \tilde\Gamma_t)$ which tends to $(f,\Gamma)$ as $t \nearrow 1^-$ but the property in Lemma \ref{Lem:GtBall} does not hold for $(\tilde f_t, \tilde \Gamma_t)$, it is not clear if our proof would carry over.

We now  give a sketch of the proof of Theorem \ref{thm}. In the discussion, let $C$ denote a positive constant depending only on the known data such as $\underline{u}, (M, g), (f, \Gamma), \psi$ and $\varphi$, but may be different from lines to lines. In particular, the constants $C$ will be independent of $t \in [0,1]$. The existence of a solution to equation \eqref{eqn} can be proved by the standard degree theory once \emph{a priori} estimates have been established, see \cite{Guan07}.
To this end, we need \emph{a priori} estimates for $u_t$ with $u_t \geq \underline{u}$
solving $\eqref{eqn}$ on $M$
up to its second order derivatives, as higher order estimates follow from Evans-Krylov's Theorem and Schauder theory.
Local interior first and second order derivative estimates has been known. (See also \cite{Duncan, DuncanNguyen21} for local pointwise second derivative estimates for strong solutions.) Since $u_t\geq \underline{u}$ in $M$ and $u_t = \underline{u} $ on $\partial M$, we obviously have
\[
\min_{M} u_t  \geq \min_{M} \underline{u} \; \mbox{and}\;
 \min_{\partial M} \nabla_{\nu} u_t \geq \min_{\partial M} \nabla_{\nu} \underline{u}.
\]
Lemma 3.1 in \cite{Guan07} yields that $\max_{M} u_t \leq C$. Then, Lemma 3.2 and  Theorem 3.3 in \cite{Guan07} ensure that
$\sup_M ( |u_t| + |\nabla u_t| ) \leq C$.
For the estimate of second order derivatives, Theorem 3.4 in \cite{Guan07} shows that
\[
\sup_M |\nabla^2 u_t| \leq C (1 + \sup_{\partial M} |\nabla^2 u_t|).
\]
We should point out that all the above mentioned results in \cite{Guan07} are valid in the present setting
as the proof only used the assumptions \eqref{f1}-\eqref{f4}. To complete this argument, it therefore remains to establish a boundary second derivative esimate, namely a bound for $\sup_{\partial M} |\nabla^2 u_t|$. Therefore, the proof of Theorem \ref{thm} reduces to that of the following result.

\begin{proposition}
	\label{prop}
	Suppose the assumptions in Theorem \ref{thm} hold.
	Then,  for any solution $u_t \in C^3 (M)$ to the Dirichlet problem \eqref{eqn} with $t \in [0,1]$ satisfying $u_t \geq \underline{u}$, we have 
\begin{equation}
\label{estimate-b}
\sup_{\partial M} |\nabla^2 u_t| \leq C,
\end{equation}
where $C$ only depends on $\underline{u}$, $(M, g)$, $(f, \Gamma)$, $\psi$ and $\varphi$.
\end{proposition}

As pointed out earlier, our main contribution concerns a bound for $\sup_{\partial M} |\nabla^2 u_t|$ which is independent of $t \in [0,1]$. The estimation of double tangential derivatives and mixed tangential-normal derivatives follows largely the arguments in \cite{Guan99, Guan07,Guan14,Guan18,GJ}. We note that $t$-dependent estimate of the double normal derivatives for general $f$ for $t < 1$ was done in \cite{LiN}. However, to achieve a uniform estimate as $t \nearrow 1^-$ requires new ideas.

As alluded to previously, we now present an example to illustrate a subtlety concerning the regularity of solutions of \eqref{conformal-a'} and \eqref{eqn}. For some $\ell > 0$ to be specified, let $(M_\ell,g)$ be the round cylinder
\begin{align*}
	M_\ell
	&= [-\ell,\ell] \times  \mathbb{S}^{n-1},\\
	g
	&= dt^2 + h,	
\end{align*}
where  $t$ is a dummy variable along the $[-\ell,\ell]$-factor and $h$ is the standard metric on $\mathbb{S}^{n-1}$. 

\begin{example}\label{Ex1}
	Let $2 \leq k \leq n$. For every $c \in \RR$, there exists $\ell > 0$ such that the problem
	\begin{equation}
		\begin{cases}
			\sigma_k(\lambda_g(W[u])) = \frac{n}{k 2^k} \tbinom{n-1}{k-1}e^{-2ku}, \quad \lambda_g(W[u]) \in \Gamma_k &\text{ in } M_\ell \setminus \partial M_\ell,\\
			u = c &\text{ on } \partial M_\ell,\\
		\end{cases}
		\label{Eq:1.3Cyl}
	\end{equation}
	admits a solution $u \in \Big(C^1(M_\ell) \cap C^\infty(M_\ell \setminus \partial M_\ell)\Big) \setminus C^2(M_\ell)$.
\end{example}

The failure of smoothness of the solution in Example \ref{Ex1} at $\partial M_\ell$
is similar to that in \cite{LiN-2012}.
See also \cite{LiN-2021} for a related phenomenon.

We point out that the solution in Example \ref{Ex1} does not lie above any smooth subsolution that we are aware of. We also do not know if it can be approximated by a sequence of smooth solutions. We conclude the introduction with two related questions on solutions of \eqref{conformal-a'} which are not necessarily bounded from below by a smooth subsolution.

\begin{question}
Assume $\psi$ is positive and smooth in $M$ and $\varphi$ is smooth on $\partial M$. Is the set of \emph{smooth} solutions to \eqref{conformal-a'} with a given $C^1$ bound bounded in $C^2(M)$?
\end{question}

Note that, by Theorem \ref{thm}, every set of smooth solutions of \eqref{conformal-a'} lying above a common smooth subsolution is bounded in $C^2(M)$.

\begin{question}
Assume $\psi$ is positive and smooth in $M$ and $\varphi$ is smooth on $\partial M$. Can every $C^1(M)$ viscosity solution to \eqref{conformal-a'} be approximated in $C^0(M)$ by $C^2(M)$ solutions?
\end{question}

The paper is organized as follows. In Section \ref{Sec2}, we introduce some useful notations, analyze the structure of the symmetric function $f$ and prove an auxiliary result which will be used to construct various barriers later on (see Lemma \ref{lemma}). In Section \ref{Sec3}, we derive the boundary estimate in Proposition \ref{prop}. In Section \ref{Sec:ENS}, we give the proof of the assertion in Example \ref{Ex1}.

\subsection*{Acknowledgments.}  Part of this work was carried out while Dong was a postdoc at Rutgers University, supported by the China Scholarship Council (File No. 201806255014). He is grateful to the Department of Mathematics at Rutgers for its warm hospitality. Li was partially supported by NSF grant DMS-2247410.

\subsection*{Rights retention statement.} For the purpose of Open Access, the authors have applied a CC BY public copyright licence to any Author Accepted Manuscript (AAM) version arising from this submission.

\section{Notations and preliminary results}\label{Sec2}

\subsection{Some facts on $(f,\Gamma)$ and $(f_t, \Gamma_t)$}
Let $\Gamma \subset \mathbb{R}^n$ be an open convex symmetric cone with vertex at the origin satisfying
$\Gamma_n \subset \Gamma \subset \Gamma_1$ and $f\in C^{\infty } (\Gamma) \cap C^0 (\bar \Gamma)$ be a symmetric function. In this subsection, we collect various properties for $(f,\Gamma)$ and $(f_t, \Gamma_t)$ which will be used in the proof of Proposition \ref{prop}. Most of these properties have appeared elsewhere. We first recall a definition in \cite{CNS}.

\begin{definition}[\cite{CNS}]\label{defGType}
$\Gamma$ is said to be of type 1 if the positive $\lambda_i$ axes belong to $\partial \Gamma$; otherwise it
is said to be of type 2.
\end{definition}

Denote by $\Gamma'$ the projection of the cone $\Gamma \subset \RR^n = \RR^{n-1} \times \RR$ onto $\mathbb{R}^{n-1}$
and by $\lambda'$ the projection of $\lambda \in \Gamma$ onto $\Gamma'$.
If $\Gamma$ is of type 1, $\Gamma'$ is an open convex symmetric cone in $\mathbb{R}^{n-1}$ and
\[
\{\lambda' \in \mathbb{R}^{n-1}: \lambda'_\alpha > 0, 1 \leq \alpha \leq n -1 \}
\subset \Gamma' \subset
\{\lambda' \in \mathbb{R}^{n-1}: \lambda'_1 + \cdots + \lambda'_{n-1} > 0\}.
\]
If $\Gamma$ is of type 2, $\Gamma' = \mathbb{R}^{n-1}$.

We note that, regardless whether $\Gamma$ is of type 1 or of type 2, the cones $\Gamma_t$ are always of type $2$ for $t \in [0,1)$. In fact, we have:

\begin{lemma}\label{Lem:GtBall}
Let $\Gamma \subset \mathbb{R}^n$ be an open convex symmetric cone with vertex at the origin satisfying
$\Gamma_n \subset \Gamma \subset \Gamma_1$. For $t \in [0,1)$, then cone $\Gamma_t$ contains the open ball of radius $\frac{1}{2n}(1-t)$ about $(0,\ldots, 0,1)$. Moreover, if $f\in C^{\infty } (\Gamma)$ is a symmetric function satisfying \eqref{f2} and \eqref{f4}, then 
\[
f_t(\lambda) \geq \frac{1}{2}(1-t)f(e)
\]
for any $\lambda$ in this ball.
\end{lemma}

\begin{proof}
For two vectors $\lambda, \tilde\lambda \in \RR^n$, we write $\lambda \geq \tilde\lambda$ if $\lambda - \tilde\lambda \in \bar\Gamma_n$. For $v \in \RR^n$ with $|v| = 1$, $r> 0$, we have
\[
(0,\ldots, 0,1) + rv \geq (-r, \ldots, - r, 1 - r)
\]
Thus, we only need to show that $\lambda := (-\frac{1}{2n}(1 - t), \ldots, - \frac{1}{2n}(1 - t), 1 - \frac{1}{2n}(1 - t)) \in \Gamma_t$ and $f_t(\lambda) \geq \frac{1}{2}(1 - t) f(e)$. Indeed, we have
\begin{align*}
t \lambda + (1-t) \sigma_1(\lambda)e
	& = (0, \ldots, 0, t) +  \frac{(1-t)(n + (n-1)t)}{2n}e\\
	&\geq \frac{(1-t)(n + (n-1)t)}{2n} e  \in \Gamma_n \subset \Gamma,
\end{align*}
which implies $\lambda \in \Gamma_t$ and
\[
f_t(\lambda) = f(t \lambda + (1-t) \sigma_1(\lambda)e) \geq \frac{1}{2}(1 - t) f(e).
\]
The proof is complete.
\end{proof}

In case $\Gamma$ is of type $1$, $\Gamma'$ is a proper subset of $\RR^{n-1}$. Let $d_{\Gamma'}$ denote the distance function to $\partial\Gamma'$. For $\lambda_0' \in \partial\Gamma'$, let $N_{\lambda_0'}(\partial\Gamma')$ denote the set of unit vectors $\gamma \in \RR^{n-1}$ such that $\Gamma'$ is contained in the half-space $\{\lambda' \in \RR^{n-1}: \gamma \cdot (\lambda' - \lambda_0') > 0\}$. We list some properties of $N_{\lambda_0'}(\partial\Gamma')$ and $d_{\Gamma'}$ which we will use later on.
\begin{itemize}
\item Since $\partial\Gamma'$ contains the ray $\RR_+\lambda_0'$, $\gamma \cdot \lambda_0' = 0$.
\item It follows that $\Gamma' \subset \{\lambda' \in \RR^{n-1}: \gamma \cdot \lambda'  > 0\}$.
\item Since $\Gamma' \supset \Gamma_n'$, this further implies that $\gamma_\alpha \geq 0$ for $1 \leq \alpha \leq n-1$.
\item For any ascending $\lambda' \in \Gamma'$,
\begin{equation}
d_{\Gamma'}(\lambda') = \inf_{\substack{\text{ascending }\lambda_0' \in \partial\Gamma'\\\text{descending } \gamma \in N_{\lambda_0'}(\partial\Gamma')}} \gamma\cdot \lambda'.
	\label{Eq:T0}
\end{equation}
\end{itemize}
The first three properties are immediate. To see the fourth property, denote the right hand side of \eqref{Eq:T0} as $\tilde d(\lambda')$. Note that $\tilde d(\lambda')$ is well-defined since there clearly exists an ascending $\lambda_0' \in \partial\Gamma'$ (e.g. $(0, \ldots, 0, 1)$) and the existence of a descending $\gamma$ in $N_{\lambda_0'}(\partial\Gamma')$ is given by \cite[Lemma 6.1]{CNS}. The inequality $\tilde d(\lambda') \geq d_{\Gamma'}(\lambda')$ is a consequence of the fact that $\Gamma'$ is supported by the hyperplane through $\lambda_0'$ and normal to $\gamma$. To see the reverse inequality, let $\lambda'_* \in \partial\Gamma'$ be such that $d_{\Gamma'}(\lambda') = |\lambda' - \lambda_*'|$. Let $\lambda'_{**}$ be an ascending rearrangement of $\lambda_*'$. By symmetricity of $\Gamma'$, we have that $\lambda_{**}' \in \partial\Gamma'$. Since $\lambda'$ is ascending, a simple induction argument on $n$ gives
\[
\lambda' \cdot \lambda_{**}' \geq \lambda' \cdot \lambda_*',
\]
which, in view of the fact that $|\lambda_*'| = |\lambda_{**}'|$, implies
\[
|\lambda' - \lambda_{**}'| \leq |\lambda' - \lambda_*'|.
\]
It follows that $d_{\Gamma'}(\lambda') = |\lambda' - \lambda_*'| = |\lambda' - \lambda_{**}'|$. Since any hyperplane at $\lambda_{**}'$ supporting $\Gamma'$ must support the ball centered at $\lambda'$ of radius $|\lambda' - \lambda_{**}'|$, we see that $N_{\lambda_{**}'}(\partial\Gamma')$ contains a unique element, namely $\gamma_{**} := \frac{\lambda' - \lambda_{**}'}{|\lambda' - \lambda_{**}'|}$. By \cite[Lemma 6.1]{CNS}, $\gamma_{**}$ is descending. It follows that 
\[
\tilde d(\lambda') \leq \gamma_{**} \cdot \lambda' = \gamma_{**} \cdot (\lambda' - \lambda_{**}') = |\lambda' - \lambda_{**}'| = d_{\Gamma}(\lambda').
\]
We have thus proved \eqref{Eq:T0}.

We next discuss certain properties of $f$ which will be used later on. The same or similar properties have appeared previously elsewhere.

\begin{definition}[\cite{Trudinger}]
	\label{def}
We say $f$ is of \emph{unbounded type} if  
\begin{equation}
\label{unbounded}
\lim_{s\rightarrow + \infty} f(\lambda_1, \ldots, \lambda_{n-1}, s) = + \infty 
\end{equation}
for every $\lambda' = (\lambda_1, \ldots, \lambda_{n-1}) \in \Gamma'$. If \eqref{unbounded} fails for some $\lambda' \in \Gamma'$, we say $f$ is of \emph{bounded type}.
\end{definition}

For example, the function $\sigma_k^{1/k} (\lambda)$ is of unbounded type for $1 \leq k \leq n$ while the quotient function $\Big( \frac{\sigma_k (\lambda)}{\sigma_\ell (\lambda)} \Big)^{ \frac{1}{k-\ell} }$ is of bounded type for $1 \leq \ell < k \leq n$.

We note that if $\Gamma$ is of type 2 and $f$ satisfies \eqref{f4}, then 
\[
\lim_{s \rightarrow + \infty} \frac{1}{s} f (\lambda', s) 
	= \lim_{s \rightarrow + \infty}   f (\frac{\lambda' }{s} , 1)  = f(0', 1) > 0,
\]
and hence $f$ is of unbounded type. A consequence of the above is that, if $f$ satisfies \eqref{f4}, then, regardless of the type of $\Gamma$ and $f$, $f_t$ is always of unbounded type for $t \in [0,1)$.

We should clarify that, by our assumption \eqref{f1} and \eqref{f3}, if, at one point $\lambda' \in \Gamma'$,
\eqref{unbounded} holds, it then holds at any $\mu' \in \Gamma'$.
A simple proof is given below.
For any $\mu' = (\mu_1, \ldots, \mu_{n-1}) \in \Gamma'$, we can choose sufficiently small $\varepsilon > 0$ such that
$\gamma' = \lambda' + (1+\varepsilon)(\mu' - \lambda')$ is also in $\Gamma'$. Then, by the concavity \eqref{f3} of $f$,
we have
\[
f(\mu_1, \ldots, \mu_{n-1}, s) \geq \frac{1}{1+\varepsilon} f(\gamma_1, \ldots, \gamma_{n-1}, s)
+\frac{\varepsilon}{1+\varepsilon} f(\lambda_1, \ldots, \lambda_{n-1}, s),
\]
for $s$ large enough.
We see that, by \eqref{f1},
\[
\lim_{s\rightarrow + \infty} f(\mu_1, \ldots, \mu_{n-1}, s) = + \infty.
\]
Hence, for every $C>0$ and every compact set $K$ in $\Gamma$, there is a positive number $R = R(C, K)$ such that
\begin{equation}
\label{Condition7}
f (\lambda_1, \ldots, \lambda_n + R) \geq C, \; \mbox{for all} \;\; \lambda\in K.
\end{equation}
The above inequality is exactly \cite[condition (7)]{CNS}
and is used by Guan (see \cite[equation (1.13)]{Guan07}) to derive the boundary $C^2$ estimate.

On the other hand, if \eqref{unbounded} fails for some $\lambda'= (\lambda_1, \ldots, \lambda_{n-1}) \in \Gamma'$,
we can define the following function on $\Gamma'$:
\begin{equation}
\label{f-infty}
f_\infty (\lambda')
= \lim_{s\rightarrow + \infty} f(\lambda_1, \ldots, \lambda_{n-1}, s),\; \mbox{for all}\; \lambda' \in \Gamma'.
\end{equation}
Evidently, $f_\infty$ is homogeneous of degree one
and $f_\infty$ is non-decreasing when $\lambda_i$ is increasing for every $1\leq i \leq n-1$.
Note that $f_\infty$ is concave in $\Gamma'$, so it is also continuous in $\Gamma'$.
Furthermore, $f_\infty$ is punctually second order differentiable almost everywhere in $\Gamma'$.

Let $\mathcal{U}$ be the set of $(n-1)\times(n-1)$ symmetric matrices whose eigenvalues belong to $\Gamma'$. When $\Gamma$ is of type $1$, we may define
\[
F_\infty (A ) : = f_\infty ( \lambda(A) ) \text{ for } A \in \mathcal{U}.
\]
For every $B \in \mathcal{U}$, let $\mathcal{N}_\infty(B)$ denote the set of symmetric $(n-1)\times (n-1)$ matrices $N$ such that 
\begin{equation}\label{F}
	\sum_{\alpha,\beta} N_{\alpha\beta} (A_{\alpha \beta} - B_{\alpha \beta}) \geq F_\infty (A) - F_\infty (B) \text{ for all } A \in \mathcal{U}.
\end{equation}
The set $\mathcal{N}_\infty(B)$ is non-empty thanks to the concavity of $F_\infty$.

\begin{lemma}\label{Lem:Finfty}
Suppose $\Gamma$ is of type $1$. The following statements hold.
\begin{enumerate}[(i)]
\item If $B \in \mathcal{U}$, then every matrix $N$ in $\mathcal{N}_\infty(B)$ is positive semi-definite.

\item The set-valued map $\mathcal{N}_\infty$ anti-monotone:
\[
\sum_{\alpha,\beta}( N_{\alpha\beta} - K_{\alpha\beta} )
(B_{\alpha \beta} - C_{\alpha \beta}) \leq 0 \text{ for all } B,C \in \mathcal{U}, N \in \mathcal{N}_\infty(B), K \in \mathcal{N}_\infty(C).
\]
\item For any $A \in \mathcal{U}$ and any set $\mathcal{V} \subset \mathcal{U}$ containing $A$,
\begin{equation}
F_\infty(A) = \inf_{B \in \mathcal{V}, N \in \mathcal{N}_\infty(B)} \Big[ F_\infty(B) + \sum_{\alpha,\beta} N_{\alpha\beta} (A_{\alpha\beta} - B_{\alpha\beta})\Big].
	\label{Eq:T3}
\end{equation}

\item For any compact set $\mathcal{K} \subset \mathcal{U}$, there exists $C = C(\mathcal{K},F_\infty)$ such that $\|N\| \leq C$ for every $N \in \mathcal{N}_\infty(B), B \in \mathcal{K}$.

\end{enumerate}
\end{lemma}

\begin{proof}
(i) The positive semi-definiteness of $N \in \mathcal{N}_\infty(B)$ follows from the monotonicity of $F_\infty$ and from applying \eqref{F} to $A = B + \xi \otimes \xi$ for arbitrary $\xi \in \mathbb{R}^{n-1}$.
 
(ii)
From \eqref{F}, we have 
\begin{align*}
\sum_{\alpha,\beta} N_{\alpha \beta}  (C_{\alpha \beta} - B_{\alpha \beta}) \geq F_{\infty} (C) - F_\infty (B),\\
\sum_{\alpha,\beta}K_{\alpha \beta}  (B_{\alpha \beta} - C_{\alpha \beta}) \geq F_{\infty} (B) - F_\infty (C).
\end{align*}
Summing up, we obtain the anti-monotonicity of $\mathcal{N}_\infty$.

(iii) The direction ``$\geq$'' of \eqref{Eq:T3} follows by taking $B = A$, while the direction ``$\leq$'' of \eqref{Eq:T3} follows from \eqref{F}.

(iv) Consider first the case that $\mathcal{K}$ contains only one element, say $B$. Fix some small $\varepsilon > 0$ such that $A = B - \varepsilon I \in \mathcal{U}$. We then see from \eqref{F} that
\[
\textrm{tr}(N) \leq \varepsilon^{-1}(F_\infty(B) - F_\infty(B - \varepsilon I)) \text{ for all } N \in \mathcal{N}_\infty(B).
\]
Recalling that $N$ is positive semi-definite, we deduce that $\mathcal{N}_\infty(B)$ is bounded, i.e. the conclusion holds when $\mathcal{K}$ is a singleton set.

Consider the general case. Clearly, $\mathcal{K}$ is covered by $\cup_{A\in \mathcal{U}} (A + \mathcal{P})$ where $\mathcal{P}$ is the set of positive definite symmetric matrices. By compactness, there exist $A_1, \ldots, A_j$ such that $\mathcal{K}$ is covered by $\cup_i (A_i + \mathcal{P})$. By anti-monotonicity of $\mathcal{N}_\infty$, we have that, for any $B \in \mathcal{K}$ and $N \in \mathcal{N}_\infty(B)$, there exist some $i$ and some $K \in \mathcal{N}_\infty(A_i)$ such that $B > A_i$ and $N \leq K$. Since each $\mathcal{N}_\infty(A_i)$ is bounded, the conclusion follows.
\end{proof}

Given $\sigma > 0$, let $\Gamma_t^\sigma = \{\lambda \in \Gamma_t: f_t(\lambda) > \sigma\}$.
By our assumptions of $f$, the level set $\partial \Gamma_t^\sigma = \{\lambda \in \Gamma_t: f_t(\lambda) = \sigma\}$
is a smooth and convex non-compact complete hypersurface in $\mathbb{R}^n$.
For $\lambda \in \Gamma_t$, let
\[
\nu_{t,\lambda} := \frac{D f_t(\lambda)}{|D f_t(\lambda)|}
\]
denote the unit normal vector to $\partial \Gamma_t^{f(\lambda)}$ at $\lambda$.

The following lemma is a variant of the important Lemma 1.9 in \cite{Guan18}.

\begin{lemma}
\label{guan}
Suppose that $(f, \Gamma)$ is as in Theorem \ref{thm}. 
Given a compact set $K \subset \Gamma$
and $\beta >0$ sufficiently small, there exists a constant $\varepsilon >0$ depending only on $\beta$, $K$ and $(f, \Gamma)$ such that for any $\mu \in K$ and $\lambda \in \Gamma_t$
with $| \nu_{t,\mu} - \nu_{t,\lambda} | > \beta$ for some $t \in [0,1]$,
\[
D f_t(\lambda) \cdot (\mu  - \lambda ) \geq f_t(\mu) - f_t(\lambda) + \varepsilon \sum_i \partial_{i}f_t(\lambda) + \varepsilon.
\]
\end{lemma}

\begin{proof} We adapt the proof of \cite[Lemma 1.9]{Guan18}. In the proof $C$ denotes a constant which varies from lines to lines but depends only on $\beta, K$ and $(f,\Gamma)$. We will assume throughout that $t \in [0,1]$ and $\beta < \sqrt{2}$.

For $\mu \in K$, let $X_{t,\mu}$ denote the set of $\lambda \in \Gamma_t$ such that $|\nu_{t,\mu} - \nu_{t,\lambda}| > \beta$.

By Taylor's theorem and the compactness of $K$, we have for all $t, r \in [0,1]$, $\mu \in K$ and $v \in \mathbb{S}^{n-1}$ that
\begin{equation}
\Big|f_t(\mu + r v) - f_t(\mu) - r |Df_t(\mu)|   \nu_{t,\mu} \cdot v \Big| \leq Cr^2.
\label{Eq:g0}
\end{equation}
In particular, by \eqref{f2} and \eqref{f5}, for every $\delta > 0$, there exists $r_0 \in (0,1)$ depending only on $K$, $\delta$ and $(f,\Gamma)$ such that for all $r \in [0,r_0]$, $\mu \in K$ and $v \in \mathbb{S}^{n-1}$ satisfying $\nu_{t,\mu} \cdot v \geq \delta$ it holds that
\begin{equation}
f_t(\mu + r v) - f_t(\mu) \geq \frac{  \delta r}{C}.
\label{Eq:g1}
\end{equation}

Now observe that, if $\lambda \in X_{t,\mu}$, then \eqref{f3} implies that
\[
0 < \nu_{t,\mu} \cdot \nu_{t,\lambda} = 1 - \frac{1}{2}|\nu_{t,\mu} - \nu_{t,\lambda}|^2 < 1 - \frac{1}{2} \beta^2.
\]
Thus, if we define a unit vector $v_{t,\mu,\lambda}$ by
\[
v_{t,\mu,\lambda} := \cos \alpha \nu_{t,\mu} - \sin \alpha \nu_{t,\lambda} \text{ with }  \alpha = \arctan(\nu_{t,\mu} \cdot \nu_{t,\lambda}) \in (0,\arctan(1 - \frac{1}{2} \beta^2)),
\]
then $v_{t,\mu,\lambda} \cdot \nu_{t,\lambda} = 0$ and
\[
 \nu_{t,\mu} \cdot v_{t,\mu,\lambda} = \cos\alpha - \sin \alpha \tan \alpha = \frac{1 - \tan^2\alpha}{(1 + \tan^2\alpha)^{1/2}} \geq \frac{\beta^2(4- \beta^2)}{2(\beta^4 - 4\beta^2 + 8)^{1/2}} > 0.
\]
Therefore, with the choice $\delta = \frac{\beta^2(4- \beta^2)}{2(\beta^4 - 4\beta^2 + 8)^{1/2}}$ in \eqref{Eq:g1}, we obtain from the concavity of $f_t$ that
\begin{align}
D f_t(\lambda) \cdot (\mu  - \lambda ) 
	&= D f_t(\lambda) \cdot (\mu + r_0v_{t,\mu,\lambda}  - \lambda )
		\geq f_t(\mu + r_0v_{t,\mu,\lambda}) - f_t(\lambda)  \nonumber\\
	& \geq f_t(\mu) - f_t(\lambda) + \frac{1}{C}
	\text{ for all }  \mu \in K, \lambda \in X_{t,\mu}.
	\label{Eq:g2}
\end{align}

Next, if we define  a unit vector $\tilde v_{t,\mu,\lambda}$ by
\[
\tilde v_{t,\mu,\lambda} :=   \sin \alpha \nu_{t,\mu} - \cos \alpha \nu_{t,\lambda},
\]
then $\tilde v_{t,\mu,\lambda} \cdot \nu_{t,\mu} = 0$ and
\[
 \nu_{t,\lambda} \cdot \tilde v_{t,\mu,\lambda} = - \cos\alpha + \sin \alpha \tan \alpha   \leq -\frac{\beta^2(4- \beta^2)}{2(\beta^4 - 4\beta^2 + 8)^{1/2}} < 0.
\]
This implies on one hand that
\[
Df_t(\lambda) \cdot \tilde v_{t,\mu,\lambda} \leq -\frac{1}{C} |Df_t(\lambda)|,
\]
and on the other hand (in view of \eqref{Eq:g0}) that
\[
\Big|f_t(\mu + r \tilde v_{t,\mu,\lambda}) - f_t(\mu) \Big| \leq Cr^2 \text{ for all }   r \in [0,1], \mu \in K, \lambda \in X_{t,\mu}.
\]
Hence, by the concavity of $f_t$,
\begin{align*}
Df_t(\lambda) \cdot (\mu - \lambda) 
	&\geq Df_t(\lambda) \cdot (\mu + r \tilde v_{t,\mu,\lambda} - \lambda) + \frac{r}{C} |Df_t(\lambda)| \nonumber\\
	&\geq f_t(\mu + r \tilde v_{t,\mu,\lambda}) - f_t(\lambda) + \frac{r}{C} |Df_t(\lambda)|\\
	&\geq f_t(\mu) - f_t(\lambda) + \frac{r}{C} |Df_t(\lambda)| - Cr^2 \text{ for all } r \in [0,1], \mu \in K, \lambda \in X_{t,\mu}.
\end{align*}
Together with \eqref{f3} and \eqref{f5}, we deduce that
\begin{equation}
Df_t(\lambda) \cdot (\mu - \lambda) 
	\geq f_t(\mu) - f_t(\lambda) + \frac{1}{C} \sum_i \partial_{i} f_t(\lambda)   \text{ for all }   \mu \in K, \lambda \in X_{t,\mu}.
	\label{Eq:g3}
\end{equation}
The conclusion follows from \eqref{Eq:g2} and \eqref{Eq:g3}.
\end{proof}

Let $W = \{W_{ij}\} $ be an $n \times n$ symmetric matrix with $\lambda (W) \in \Gamma$.
Define $F_t(W) = f_t (\lambda(W))$ and $F_t^{ij}(W) = \frac{\partial F_t}{\partial W_{ij}} (W)$. Note that 
\begin{equation}\label{FWW}
\sum_l F_t^{ij}(W) W_{il} W_{jl}  = \sum_i \partial_i f_t(\lambda) \lambda_i^2,
\end{equation}
where $\lambda=(\lambda_1, \ldots, \lambda_n)$ are the eigenvalues of $W$.
The following lemma is also needed.

\begin{lemma}\cite[Proposition 2.19]{Guan14}
	\label{guan3}
	Suppose that $(f, \Gamma)$ is as in Theorem \ref{thm}. 
	Let $t \in [0,1]$. There is an index $r$ such that
	\[
	\sum_{\ell < n} F_t^{ij}(W) W_{i\ell} W_{j\ell} \geq \frac{1}{2} \sum_{i\neq r} \partial_i f_t(\lambda) \lambda_i^2.
	\]
\end{lemma}

To control $\sum_i \partial_i f_t(\lambda) |\lambda_i|$, 
we will frequently use the following inequality
\begin{equation}\label{f-lambda}
\sum_i \partial_i f_t(\lambda) |\lambda_i| \leq \varepsilon \sum_i \partial_i f_t (\lambda)\lambda_i^2 + \frac{1}{\varepsilon} \sum_i \partial_i f_t(\lambda)
\end{equation}
for any $\varepsilon > 0$,
which is obtained by Cauchy-Schwarz inequality. We will also need the following stronger inequality.
\begin{lemma}\cite[Corollary 2.21]{Guan14}
\label{guan2}
Suppose that $(f, \Gamma)$ is as in Theorem \ref{thm}. Let $t \in [0,1]$. 
For any $\lambda \in \Gamma$, index $r$ and $\varepsilon > 0$, we have
\[
\sum_i \partial_i f_t(\lambda) |\lambda_i| \leq \varepsilon \sum_{i \neq r} \partial_i f_t (\lambda)\lambda_i^2 + \frac{C}{\varepsilon} \sum_i \partial_i f_t(\lambda) + C,
\]
where $C$ depends only on $n$ and $f_t(\lambda)$.
\end{lemma}

\subsection{Some estimates for the linearized operator}\label{ssec22}
We next gather estimates for the linearized operator when acting on various auxiliary functions which we will use later on. We begin by introducing some notations which will be frequently used. Let $d(x)$ be the distance function from $x\in M$ to the boundary $\partial M$.
For $x_0 \in \partial M$, let $\rho(x) = d (x, x_0)$ be the distance from $x$ to $x_0$ and
\[
\Omega_{\delta, x_0} = \{x \in M: \rho(x) < \delta\}
\]
be a $\delta$-neighborhood of $x_0$.
Around a point $x \in \partial M$, we shall always choose smooth orthonormal local frame $e_1, \ldots, e_n$
by parallel transporting a local orthonormal frame $e_1, \ldots, e_{n-1}$ on $\partial M$ and the inward unit normal $e_n$
to $\partial M$ along geodesics perpendicular to $\partial M$.
Under a local orthonormal frame $e_1, \ldots, e_n$ on $M$, we denote $\nabla_i = \nabla_{e_i}$ and
$\nabla_{ij} = \nabla_{i} \nabla_{j}$. For a given solution $u_t$ of \eqref{eqn}, let us introduce the following notations:
\begin{align*}
W^t[u_t] 
	&:= t W[u_t] + (1-t) \textrm{tr}(g^{-1}W[u_t]) g, \\
F_t(W[u_t])
	&:= f_t (\lambda_g (W[u_t])) = f(\lambda_g(W^t[u_t])),\\
F_t^{ij}
	&:= \frac{\partial F_t}{\partial W_{ij}} (W[u_t]),\\
\cT_t
	&:= \sum_i F_t^{ii}.
\end{align*}
We will also use $\lambda^{(t)} = (\lambda^{(t)}_1, \ldots, \lambda^{(t)}_n)$ and $\eta^{(t)} = (\eta^{(t)}_1, \ldots, \eta^{(t)}_n)$ to denote the eigenvalues of $W[u_t]$ and $W^t[u_t]$ with respect to $g$, both arranged in ascending order. For $v\in C^2 (M)$, the linearized operator of equation \eqref{eqn} at $u_t$ is defined as
\[
\mathcal{L}_t v := F_t^{ij} \Big( \nabla_{ij} v + \nabla_i u_t \nabla_j v + \nabla_j u_t \nabla_i v - \sum_l \nabla_l u_t \nabla_l v \delta_{ij} \Big),
\]
where Einstein summation convention is used whenever the same index is repeated in a subscript and a superscript positions.

\begin{lemma}\label{Lem:bbb}
Suppose that $(f, \Gamma)$, $u_t$ and $\underline{u}$ are as in Theorem \ref{thm}. Let $X$ be a smooth vector field on a compact set $K \subset M$. Then there exists $C > 0$ depending only on $\|X\|_{C^2(K)}$, $\|u_t\|_{C^1 (M)}, (M, g)$, $(f, \Gamma)$, $\underline{u}$ and $\psi$ such that
\begin{equation}
	\label{bbb}
	\mathcal{L}_t (X (u_t - \underline{u}))
	\leq C \Big(1 + \cT_t + \sum_i \partial_i f_t(\lambda^{(t)})|\lambda_i^{(t)}|\Big) \text{ in } K.
\end{equation}
\end{lemma}

\begin{proof}
Note that, by the product rule and
commuting derivatives,
\begin{multline}
	\label{formula}	\nabla_{ij} (X(u_t))
		= \nabla_{ij} g (\nabla u_t, X)  \\
		 =  X ( \nabla^2 u_t ( e_i, e_j ) ) +  \nabla^2 u_t ([e_i, X], e_j) + \nabla^2 u_t (e_i, [e_j, X]) +  g(\nabla u_t, T_{ij}(X)),
\end{multline}
where
\[
T_{ij}(X) = \nabla_{[e_i, X]} e_j + \nabla_{e_i} [e_j, X] + [X, \nabla_{e_i} e_j].
\] 
Hence, since $F_t^{ij} X(W_{ij} [u_t]) = X(\psi (x, u_t))$, we have
\begin{equation}
\mathcal{L}_t (X(u_t)) \leq C \big(1+ \cT_t \big) + 2   F_t^{ij} \nabla^2 u_t (e_i, [e_j, X]).
	\label{LtXut}
\end{equation}

To proceed, we note that, by the chain rule and the fact that 
$F_t(W) = F_t (AWA^T)$
for any constant orthogonal matrix $A$,
\[
F_t^{ij}(W)   =   A_{pi} F_t^{pq} (AWA^T) A_{qj}
\]
and 
\[
F_t^{ij}(W)  W_{j \ell} = \sum_j A_{pi} F_t^{pq} (AWA^T) A_{qj} W_{j \ell}
= \sum_s A_{pi} F_t^{pq}(AWA^T) [AWA^T]_{qs} A_{s \ell}.
\]
In particular, if $A$ is the orthogonal matrix so that $AWA^T = \text{diag} (\lambda_1, \dots, \lambda_n)$, then $F_t^{pq} (AWA^T) = \partial_p f_t(\lambda) \delta_{pq}$,
and the above identity implies
\[
|F_t^{ij} (W) W_{j \ell}| \leq C \sum_i \partial_i f_t(\lambda) |\lambda_i|.
\]
It follows that
\begin{equation}
	\label{b}
	F_t^{ij} \nabla^2 u_t (e_i, [e_j, X] ) \leq C \cT_t + C \sum_i \partial_i f_t(\lambda^{(t)})|\lambda^{(t)}_i|.
\end{equation}
Returning to \eqref{LtXut}, we deduce
\begin{equation}
	\mathcal{L}_t (X (u_t))
	\leq C \Big(1 + \cT_t + \sum_i \partial_i f_t(\lambda^{(t)})|\lambda_i^{(t)}|\Big).\label{bbbx}
\end{equation}
The conclusion is readily seen.
\end{proof}

\begin{lemma}
\label{lem:LtNu2}
Suppose that $(f, \Gamma)$, $u_t$ and $\underline{u}$ are as in Theorem \ref{thm}. Then there exists $C > 0$ depending only on $\|u_t\|_{C^1 (M)}, (M, g)$, $(f, \Gamma)$,  $\underline{u}$ and $\psi$ such that
\begin{equation}\label{LDu^2}
	\mathcal{L}_t (|\nabla u_t |^2 - |\nabla \underline{u} |^2)
	\geq - C \Big( 1+ \cT_t + \sum_i \partial_i f_t(\lambda^{(t)})|\lambda^{(t)}_i| \Big)  + \sum_i \partial_i f_t(\lambda^{(t)}) (\lambda^{(t)}_i)^2.
\end{equation}
\end{lemma}

\begin{proof}
We have
\[
\mathcal{L}_t (|\nabla u_t |^2 - |\nabla \underline{u} |^2)
	\geq 2\sum_l \nabla_l u_t \mathcal{L}_t \nabla_l u_t + 2 \sum_l F_t^{ij} \nabla_{i} \nabla_l u_t \nabla_{j} \nabla_l u_t -  C \big( 1+ \cT_t \big).
\]
By Cauchy-Schwarz inequality, we see
\[
2 \sum_l F_t^{ij} \nabla_{i} \nabla_l u_t \nabla_{j} \nabla_l u_t \geq \sum_l F_t^{ij} W_{il}[u_t] W_{jl}[u_t] - C \cT_t.
\]
The conclusion is readily seen from \eqref{FWW}, \eqref{bbbx} and the above two estimates.
\end{proof}

\begin{lemma}
\label{lemma}
Suppose that $(f, \Gamma)$, $u_t$ and $\underline{u}$ are as in Theorem \ref{thm}.
Then there exists a constant $\varepsilon > 0$ such that,
for sufficiently small $\delta  > 0$,
the function
\[
v= (u_t-\underline{u}) - \frac{(u_t-\underline{u})^2}{2} + \delta d -  d^2
\]
satisfies
\[
\mathcal{L}_t v \leq - \varepsilon \big( 1 + \cT_t \big)\; \mbox{in} \; \Omega_{\delta, x_0},\;
v\geq 0 \; \mbox{on} \; \bar{\Omega}_{\delta,x_0}\; \mbox{and}  \; v|_{\partial M} = 0,
\]
where $\varepsilon$,  $\delta$ depend only on $\|u_t\|_{C^1 (M)}, (M, g)$, $(f, \Gamma)$ and $\underline{u}$.
\end{lemma}

\begin{proof} By a direct calculation, we have that
\[\begin{aligned}
& \mathcal{L}_t (u_t - \underline{u}) \\
= &\; F_t^{ij} (W_{ij}[u_t] - W_{ij} [\underline{u}]) \\
&\; + F_t^{ij} \Big( \nabla_i \underline{u} \nabla_j \underline{u} -\frac{1}{2} |\nabla \underline{u}|^2 \delta_{ij}
 - \nabla_i u_t \nabla_j u_t + \frac{1}{2}|\nabla u_t|^2 \delta_{ij} \Big)\\
&\; + F_t^{ij} \Big( \nabla_i u_t \nabla_j (u_t - \underline{u}) + \nabla_j u_t \nabla_i (u_t - \underline{u})
  - \sum_l \nabla_l u_t \nabla_l (u_t - \underline{u}) \delta_{ij} \Big)\\
  = &\; F_t^{ij} (W_{ij}[u_t] - W_{ij} [\underline{u}]) + F_t^{ij} \nabla_i (u_t - \underline{u}) \nabla_j (u_t - \underline{u})
   - \frac{1}{2} |\nabla (u_t - \underline{u})|^2 \cT_t.
\end{aligned}\]
Therefore, we obtain that
\[\begin{aligned}
& \mathcal{L}_t \Big( (u_t - \underline{u}) - \frac{(u_t - \underline{u})^2}{2}\Big)\\
= &\; \Big(1- (u_t - \underline{u}) \Big) \mathcal{L}_t (u_t - \underline{u}) - F_t^{ij} \nabla_i (u_t - \underline{u}) \nabla_j (u_t - \underline{u})\\
= &\;  \Big(1- (u_t - \underline{u}) \Big) F_t^{ij} (W_{ij}[u_t] - W_{ij} [\underline{u}]) - (u_t - \underline{u}) F_t^{ij} \nabla_i (u_t - \underline{u}) \nabla_j (u_t - \underline{u})\\
 &\; - \frac{1}{2} \Big( 1- (u_t - \underline{u}) \Big)  |\nabla (u_t - \underline{u})|^2 \cT_t.
\end{aligned}\]

Note that $F_t^{ij} (W_{ij}[u_t] - W_{ij}[\underline{u}] ) \leq 0$ since $F_t$ is concave, $\underline{u} \leq u_t$ and $\psi_z \leq 0$, i.e.
\[
0 \leq \psi (x, \underline{u}) - \psi (x, u_t) \leq F_t (W[\underline{u}]) - F_t(W[u_t]) \leq  F_t^{ij} (W_{ij}[\underline{u}] - W_{ij}[u_t] ). 
\]
Also, in $\Omega_{\delta, x_0}$ with $\delta$ small enough depending on $\|u_t\|_{C^1(M)}$ and $\underline{u}$, it is easy to see that
\[
0 \leq u_t - \underline{u} \leq \frac{1}{2}.
\] 
Combining the above two facts and that $\{ F_t^{ij} \}$ is positive definite, we then arrive at
\[
 \mathcal{L}_t \Big( (u_t - \underline{u}) - \frac{(u_t - \underline{u})^2}{2}\Big) \leq 
 \frac{1}{2} F_t^{ij} (W_{ij}[u_t] - W_{ij}[\underline{u}] ). 
\]
Let $\delta > 0$ be small enough depending on $(M, g)$ so that $d$ is regular in $\Omega_{\delta, x_0}$.
Due to that $\nabla_\alpha d = 0$ for $1\leq \alpha \leq n-1$ and $\nabla_n d = 1$, we obtain that
\[
\mathcal{L}_t (\delta d - d^2) = (\delta - 2d) \mathcal{L}_t d - 2 F_t^{nn}
\leq  C  \delta  \cT_t - 2 F_t^{nn}.
\]
Therefore, in $\Omega_{\delta, x_0}$ with $\delta$ small enough depending on $\|u_t\|_{C^1(M)}$, $(M, g)$ and $\underline{u}$, we have
\begin{equation}
\label{Lv}
\mathcal{L}_t v \leq \frac{1}{2} F_t^{ij} (W_{ij}[u_t] - W_{ij}[\underline{u}] ) + C \delta \cT_t - 2 F_t^{nn},
\end{equation}
where $C$ depends on $(M, g)$ and $\|u_t\|_{C^1 (M)}$.

Let $\lambda (x) := \lambda^{(t)}(x) = \lambda (W [u_t](x))$ and $\mu (x) := \lambda (W [\underline{u}](x))$. We assume that the eigenvalues $ \lambda ( x )$ and $ \mu ( x )$ are arranged in ascending order:
\[
\lambda_1 \leq \cdots \leq \lambda_n \; \mbox{and} \;
\mu_1 \leq \cdots \leq \mu_n.
\]
Therefore, by \eqref{f3} and the fact that $f$ is symmetric, we can derive that 
\[
\partial_1 f_t (\lambda) \geq \cdots \geq \partial_n f_t (\lambda).
\]
We then have
\begin{equation}
\label{F-f-1}
F_t^{ij} (W_{ij}[\underline{u}] - W_{ij}[u_t]) = F_t^{ij}W_{ij}[\underline{u}]  - \sum_i \partial_i f_t (\lambda) \lambda_i \geq \sum_i \partial_i f_t(\lambda) (\mu_i - \lambda_i),
\end{equation}
where for the last inequality we used \cite[Lemma 6.2]{CNS} (or see \cite[Lemma 1.5]{Spruck}).
Since $\{\nu_{t,\mu(x)}: t \in [0,1], x\in M \}$ is a compact set in $\Gamma_n$,
there exists a sufficiently small constant $\beta > 0$
depending on $\underline{u}$, $(M, g)$ and $(f, \Gamma)$
such that for all $t \in [0,1]$, $x \in M$
\[
\nu_{t,\mu(x) } - 2\beta (1, \ldots, 1) \in \Gamma_n.
\]

We continue our argument depending on whether $ |\nu_{t,\mu(x)} - \nu_{t,\lambda(x)}|$ is at most $\beta$ or more than $\beta$.

If $ |\nu_{t,\mu(x)} - \nu_{t,\lambda(x)}| \leq \beta $, we have
$\nu_{t,\lambda(x)} - \beta (1, \ldots, 1) \in \Gamma_n$,
which is equivalent to that, at the point $x$,
\begin{equation}
\label{F-f-3}
\partial_i f_t (\lambda ) \geq \beta |D f_t(\lambda)| \geq \frac{\beta}{\sqrt{n}} \sum_i \partial_i f_t(\lambda ), \; \forall \; 1\leq i\leq n.
\end{equation}
Hence, combining with \eqref{Lv} and \eqref{F-f-1}, we have
\[
\mathcal{L}_t v \leq C \delta \cT_t- \frac{2\beta}{\sqrt{n}} \cT_t.
\]
 
If $|\nu_{t,\mu(x)} - \nu_{t,\lambda(x)}| > \beta$, then according to Lemma \ref{guan}
 and since $\mu(M)$ is a compact subset of $\Gamma$, there exists a positive constant $\varepsilon$ depending on $\beta$,
$(f, \Gamma)$ and $\underline{u}$ such that
\begin{equation}
\label{F-f-2}
\sum_i \partial_i f_t(\lambda) (\mu_i - \lambda_i )  \geq \varepsilon \sum_i \partial_i f_t(\lambda ) + \varepsilon,
\end{equation}
where we have used $f_t(\mu (x)) \geq f(\mu(x)) \geq \psi(x,\underline{u}(x)) \geq \psi(x,u_t(x)) = f_t(\lambda (x))$. Combining with \eqref{Lv}, \eqref{F-f-1} and \eqref{F-f-2}, we have
\[
\mathcal{L}_t v \leq - \frac{\varepsilon}{2} \big( \cT_t + 1 \big) + C  \delta  \cT_t.
\]

In both cases, we can choose  $\delta$ small which only depend on $\|u_t\|_{C^1 (M)}$, $(M, g)$, $\varepsilon$ and $\beta$ to prove Lemma \ref{lemma} for a small $\varepsilon > 0$ since $\cT_t > f(1, \ldots, 1)$.
\end{proof}

\section{A priori boundary second derivative estimates}\label{Sec3}

Throughout this section, $C$ denotes some generic constant that may change from lines to lines but depends only on $\|u_t\|_{C^1 (M)}, (M, g), (f, \Gamma), \underline{u}, \psi$ and $\varphi$. Proposition \ref{prop} follows immediately from the following two lemmas.

\begin{lemma}
\label{lem:p1}
Under the assumptions of Proposition \ref{prop}, for any point $x_0 \in \partial M$ and any adapted frame $e_1, \ldots, e_n$ near $x_0$ as in Subsection \ref{ssec22}, it holds that
\[
|\nabla_{ij} u_t(x_0)| \leq C \text{ for } t \in [0,1], (i,j) \neq (n,n).
\]
\end{lemma}

\begin{lemma}
\label{lem:p2}
Under the assumptions of Proposition \ref{prop}, for any point $x_0 \in \partial M$ and any adapted frame $e_1, \ldots, e_n$ near $x_0$ as in Subsection \ref{ssec22}, it holds that
\[
|\nabla_{nn} u_t(x_0)| \leq C \text{ for } t \in [0,1].
\]
\end{lemma}

The proof of Lemma \ref{lem:p1} is more or less identical to an argument in \cite{Guan07}. The proof of Lemma \ref{lem:p2} builds on existing arguments e.g. in \cite{Guan07, Guan14}, but, as pointed out earlier, requires new ideas to obtain uniform estimate as $t \nearrow 1^-$.

\subsection{Double tangential and mixed normal-tangential derivative estimates}

Recall the notations $W^t[u_t], F_t, F_t^{ij}, \cT_t, \lambda^{(t)}, \eta^{(t)}$ introduced in Subsection \ref{ssec22}.
 
\begin{proof}[Proof of Lemma \ref{lem:p1}] 

First, since $u_t - \underline{u} = 0$ on $\partial M$, we have
\[
\nabla_{\alpha\beta} (u_t - \underline{u})(x_0) = - \nabla_n (u_t - \underline{u})(x_0) \Pi (e_\alpha, e_\beta), \; 1\leq \alpha, \beta \leq n-1,
\]
where $\Pi$ denotes the second fundamental form of $\partial M$ with respect to $-e_n$. We therefore obtain the estimate for the pure tangential second order derivatives: $
|\nabla_{\alpha\beta} u_t (x_0)| \leq C$ for $t \in [0,1]$ and $1\leq \alpha, \beta \leq n-1$.

In the rest of the proof, we show that $|\nabla_{\alpha n} u_t (x_0)| \leq C$ for $t \in [0,1]$ and $1\leq \alpha \leq n-1$. Given $1 \leq \alpha \leq n-1$, define
\[
  w = \pm \nabla_\alpha (u_t - \varphi) - \sum_{\ell < n} |\nabla_\ell (u_t - \varphi)|^2,
\]
where $\varphi$ is extended to $M$ with $\nabla_n \varphi = 0$ on $\partial M$. Note that $w = 0$ on $\partial M$ near $x_0$.
By \eqref{bbb} in Lemma \ref{Lem:bbb},
a direct calculation shows that
\begin{align*}
 & \mathcal{L}_t  \sum_{\ell < n} |\nabla_\ell (u_t - \varphi)|^2 \\
  &\qquad= 2 \sum_{\ell< n} \Big( \nabla_\ell ( u_t -  \varphi) \mathcal{L}_t \nabla_\ell ( u_t - \varphi) + F_t^{ij} \nabla_i \nabla_\ell ( u_t -  \varphi) \nabla_j \nabla_\ell ( u_t - \varphi) \Big)\\
 &\qquad\geq  - C \Big( 1 + \cT_t + \sum_i \partial_i f_t(\lambda^{(t)}) |\lambda^{(t)}_i| \Big) +  \sum_{\ell < n} F_t^{ij} W_{i\ell} [u_t] W_{j\ell} [u_t].
\end{align*}
Using once again \eqref{bbb} and appealing to Lemma \ref{guan3}, we deduce that
\[
\mathcal{L}_t w \leq  C \Big( 1+ \cT_t  + \sum_i \partial_i f_t(\lambda^{(t)}) |\lambda^{(t)}_i| \Big)  - \frac{1}{2} \sum_{i\neq r} \partial_i f_t(\lambda^{(t)}) (\lambda^{(t)}_i)^2
\]
for some index $r$.
By Lemma \ref{guan2} and choosing $\varepsilon$ small enough, we finally arrive at
\begin{equation}
\label{Lw}
\mathcal{L}_t w \leq  C \big( 1+ \cT_t \big).
\end{equation}

Now, define $h = w + B \rho^2 + A v$, where $\rho$ and $v$ are as in Subsection \ref{ssec22}. Note that $h(x_0) = 0$, $h|_{\partial M \cap \partial \Omega_{\delta, x_0}} \geq 0$
and $h|_{\partial\Omega_{\delta, x_0} \setminus \partial M} \geq 0$ as long as $B$ large enough depending on $\|u_t\|_{C^1(M)}$, $\varphi$ and $\delta$. By Lemma \ref{lemma} and choosing $A\gg B\gg 1$, where $A$ depends on $\|u_t\|_{C^1(M)}, (M, g)$, $\underline{u}$, $\varphi$, $\psi$, $B$ and $\varepsilon$,
we see that
\[
\mathcal{L}_t h \leq ( C  + C B - A \varepsilon ) \big( 1+ \cT_t \big) \leq 0.
\]
By maximum principle, we have that $h \geq 0$ in $M \cap \Omega_{\delta, x_0}$ and $\nabla_n h (x_0) \geq 0$, which implies that
\[
|\nabla_{\alpha n} u_t (x_0)| \leq C \text{ for }  t \in [0,1] \text{ and  } 1\leq \alpha \leq n-1,
\]
which concludes the proof.
\end{proof}

\subsection{Double normal derivative estimates}

As an immediate consequence of Lemma \ref{lem:p1} and the fact that $\Gamma_t \subset \Gamma_1$, we have:

\begin{corollary}\label{Cor:X}
Under the assumptions of Proposition \ref{prop}, for any point $x_0 \in \partial M$ and any adapted frame $e_1, \ldots, e_n$ near $x_0$ as in Subsection \ref{ssec22}, it holds that
\[
\nabla_{nn} u_t(x_0) \geq -C \text{ for } t \in [0,1].
\]
\end{corollary}

In view of Corollary \ref{Cor:X}, in order to prove Lemma \ref{lem:p2}, we only need to give an upper bound $\nabla_{nn} u_t$ independent of $t$. Before doing this, we note a weaker estimate for $\nabla_{nn} u$ which degenerates as $t \rightarrow 1$ but is sufficient to imply a uniform bound for $W^t_{ij}[u_t]$ as $t \rightarrow 1$ when $(i,j) \neq (n,n)$. Although simpler than most other estimates in this paper, this estimate plays an important role in our subsequent argument.

\begin{lemma}
\label{lem:px}
Under the assumptions of Proposition \ref{prop}, for any point $x_0 \in \partial M$ and any adapted frame $e_1, \ldots, e_n$ near $x_0$ as in Subsection \ref{ssec22}, it holds that
\[
\nabla_{nn} u_t(x_0) \leq \frac{C}{1-t}\text{ for } t \in [0,1).
\]
Consequently, 
\[
|W_{ij}^t[u_t](x_0)| \leq C \text{ for all } t \in [0,1] \text{ and } (i,j) \neq (n,n).
\]
\end{lemma}

\begin{proof} We use an argument in \cite{LiN} with the help of Lemma \ref{Lem:GtBall}. By this lemma and the Lipschitz continuity of ordered eigenvalues of symmetric matrices (Weyl's inequality), if a symmetric matrix $B = (b_{ij})$ satisfies $b_{nn} = 1$ and $|b_{ij}| \leq \frac{1-t}{C}$ for $(i,j) \neq (n,n)$, then $\lambda(B) \in \Gamma_{t}$ and $f_t(\lambda(B) \geq \frac{1-t}{C}$ for all $t < 1$. Thus if $W_{nn}[u_t](x_0) \geq \frac{C}{1 - t}$, then (in view of the estimate in Lemma \ref{lem:p1}),
\[
C \geq \psi(x_0,u_t) = F_t(W[u_t]) = W_{nn}[u_t] f_t\Big(\lambda\Big(\frac{W[u_t]}{W_{nn}[u_t]}\Big)\Big) \geq \frac{1-t}{C} W_{nn}[u_t].
\]
It follows that $W_{nn}[u_t](x_0) \leq \frac{C}{1 - t}$ and hence $\nabla_{nn} u_t (x_0) \leq \frac{C}{1 - t}$. It also follows that $\mathrm{tr}(g^{-1} W[u_t])(x_0) \leq \frac{C}{1-t}$. The last assertion is readily seen from this statement together with the definition of $W^t[u_t]$ and Lemma \ref{lem:p1}.
\end{proof}

We next observe that if $\Gamma$ is of type $2$ then there exists $C > 0$ such that if a matrix $B = (b_{ij})$ satisfies $b_{nn} = 1$ and $|b_{ij}| \leq \frac{1}{C}$, then $\lambda(B) \in \Gamma \subset \Gamma_t$ and $f_t(\lambda(B)) \geq \frac{1}{C}$ for all $t \leq 1$. The proof of Lemma \ref{lem:px} above can then be applied yielding:

\begin{lemma}
\label{lem:p2-t2}
If $\Gamma$ is of type 2 (in the sense of Definition \ref{defGType}), then Lemma \ref{lem:p2} holds.
\end{lemma}

In the rest of this section, we consider the proof of Lemma \ref{lem:p2} when $\Gamma$ is not of type $2$, i.e. $\Gamma$ is of type $1$, and for $t$ close to $1$. In particular, $\Gamma'$, the projection of $\Gamma$ onto $\RR^{n-1}$, is a proper subset of $\RR^n$. Recall that $d_{\Gamma'}$ denotes the distance function to $\partial\Gamma' \neq \emptyset$. In view of Lemma \ref{lem:px}, the task of bounding the double normal derivatives of $u$ reduces to the task of bounding $\eta^{(t)}_n$. 

We borrow an idea from the work \cite[Section 6]{CNS} for Hessian equations, namely we aim to show first that the projection $(\eta^{(t)})' (x) :=   (\eta^{(t)}_1 (x), \ldots, \eta^{(t)}_{n-1} (x))$ of $\eta^{(t)}(x)$ onto $\Gamma'$ stays in a compact subset of $\Gamma'$. (This idea was also employed in \cite{Guan07} for the $\sigma_k$-cases; see Lemma 2.4 therein.) Once this is done, we follow the line of argument in \cite{Trudinger} and split the argument according to whether $f$ is of unbounded type or of bounded type. 
 
For some computational advantage, it is more convenient to control a related object that is asymptotically the same as $(\eta^{(t)})'$ when $\nabla_{nn} u_t$ becomes large. More precisely, let $\tilde W^t[u_t] (x)$ denote the restriction of $W^t[u_t]$ to $T_x \partial M$ 
and let $\tilde\eta^{(t)} (x)$ 
denote the eigenvalues of $\tilde W^t [u_t](x)$ with respect to the induced metric of $g$ on $\partial M$
where $\tilde \eta^{(t)}_1 \leq \cdots \leq \tilde \eta^{(t)}_{n-1}$. By Lemma \ref{lem:px}, $|\tilde \eta^{(t)}| \leq C$ on $\partial M$. By Cauchy's interlacing theorem, $\tilde \eta^{(t)}_\alpha -  \eta^{(t)}_\alpha \geq 0$ for $1 \leq \alpha \leq n-1$ and so $\tilde \eta^{(t)}\in \Gamma'$. Observe that, by \cite[Lemma 1.2]{CNS} and Lemma \ref{lem:px}, we know that if $ \eta^{(t)}_n(x)$
goes to infinity at a point $x$, then 
\begin{equation}\label{eigen-1}
 \eta^{(t)} _\alpha(x) = \tilde \eta^{(t)}_\alpha(x) + o(1), \; 1\leq \alpha \leq n-1,
\end{equation}
where the implicit bound in the little $o$-term depends only on the constant $C$ in Lemma \ref{lem:px}.

We prove:
\begin{lemma}
\label{key}
Under the hypotheses of Proposition \ref{prop}, suppose in addition that $\Gamma$ is of type 1. There exists a constant $c_0 \in (0,1)$ depending on 
$\|u_t\|_{C^1(M)}$, $\|\nabla_{\alpha\beta} u_t\|_{C^0 (\partial M)}$, $\|\nabla_{\alpha n} u_t\|_{C^0 (\partial M)}$, $(f, \Gamma)$, $(M, g)$, $\underline{u}$, $\varphi$ and $\psi$,
such that
\[
d_{\Gamma'} (\tilde \eta^{(t)}   (x)) \geq c_0 \; \mbox{on}\; \partial M \text{ for } t \in [1-c_0,1].
\]
\end{lemma}

\begin{proof}
Recall from Section \ref{Sec2} that, for $\lambda_0' \in \partial\Gamma'$, $N_{\lambda_0'}(\partial\Gamma')$ denotes the set of unit vectors $\gamma \in \RR^{n-1}$ such that $\Gamma'$ is contained in the half-space $\{\lambda' \in \RR^{n-1}: \gamma \cdot (\lambda' - \lambda_0') > 0\}$. For $\lambda_0' \in \partial\Gamma'$ and $\gamma \in N_{\lambda_0'}(\partial\Gamma')$, define
\[
\zeta_{t,\lambda_0',\gamma} := \gamma\cdot \tilde \eta^{(t)} - (1-t)  \sum_{\alpha < n} \gamma_\alpha \nabla_{nn} u_t.
\]
We will show that there exists $c_0 \in (0,1)$ depending on 
$\|u_t\|_{C^1(M)}$, $\|\nabla_{\alpha\beta} u_t\|_{C^0 (\partial M)}$, $\|\nabla_{\alpha n} u_t\|_{C^0 (\partial M)}$, $(f, \Gamma)$, $(M, g)$, $\underline{u}$, $\varphi$ and $\psi$,
such that for all ascending $\lambda_0' \in \partial\Gamma'$ and descending $\gamma \in N_{\lambda_0'}(\partial\Gamma')$,
\begin{equation}
\zeta_{t,\lambda_0',\gamma}(x) \geq 2c_0 \; \mbox{on}\; \partial M \text{ for } t \in [1-c_0,1].
	\label{Eq:T1}
\end{equation}
Once this is done, since $\nabla_{nn}u_n \geq - C$ (by Corollary \ref{Cor:X}) and $\gamma_\alpha \geq 0$, it follows that, after possibly shrinking $c_0$ slightly, for all ascending $\lambda_0' \in \partial\Gamma'$ and descending $\gamma \in N_{\lambda_0'}(\partial\Gamma')$,
\[
\gamma\cdot \tilde \eta^{(t)}(x) \geq 2c_0 + (1-t)  \sum_{\alpha < n} \gamma_\alpha \nabla_{nn} u_t \geq c_0\; \mbox{on}\; \partial M \text{ for } t \in [1-c_0,1], 
\]
which then implies the conclusion in view of \eqref{Eq:T0}.

Fix some ascending $\lambda_0' \in \partial\Gamma'$ and descending $\gamma \in N_{\lambda_0'}(\partial\Gamma')$. To lighten up the notation, we will abbreviate $\zeta_{t,\lambda_0',\gamma}$ to $\zeta_t$. Pick a point $y_0 \in \partial M$ such that $\zeta_t$ attains its minimum on $\partial M$ at $y_0$. We choose a local orthonormal frame $e_1, \ldots, e_n$ around $y_0$ as before such that
$\tilde W[u_t](y_0) = \{W_{\alpha\beta} [u_t] (y_0)\}_{1\leq \alpha, \beta \leq n-1}$ is diagonal
and
\[
W_{11}[u_t] (y_0) \leq \cdots \leq W_{(n-1)(n-1)}[u_t](y_0).
\]
Our aim is to prove $|\nabla_{nn} u_t (y_0)| \leq C$ for $t \in [1 - c_0,1]$. By Corollary \ref{Cor:X}, we only need to show $\nabla_{nn} u_t(y_0) \leq C$. Once this is done, estimate \eqref{Eq:T1} is established as follows: We know that $\eta^{(t)}(y_0) = \lambda (W^t[u_t] (y_0) )$ belongs to a compact subset $K$ of $\Gamma$ which depends only on $\|u_t\|_{C^1 (M)}$, $(f, \Gamma)$, $\psi$, the bounds established in Lemma \ref{lem:px} as well as the bound of $|\nabla_{nn} u (y_0)|$.
Since the projection $K'$ of $K$ onto $\Gamma'$ is a compact subset of $\Gamma'$, we have that 
$d_{\Gamma'} ((\eta^{(t)})'( y_0) ) ) \geq 3c_0$ for some $c_0 > 0$ depending only on $K$ and $\Gamma$.
By Cauchy's interlacing theorem, we have $q:= \tilde \eta^{(t)}  (y_0) - (\eta^{(t)})' (y_0) \geq 0$. Now if $p \in \partial \Gamma'$ is such that $d_{\Gamma'} ( \tilde\eta^{(t)}  (y_0) ) = |\tilde\eta^{(t)}  (y_0)) - p|$, then $p- q \notin \Gamma'$ and so
\begin{align*}
3c_0 \leq d_{\Gamma'} ( (\eta^{(t)} )' (y_0) ) 
&\leq  | (\eta^{(t)} )' (y_0)  - ( p - q ) | \\
&=  | \tilde \eta^{(t)}  (y_0)  - p| = 
d_{\Gamma'} ( \tilde\eta^{(t)}  (y_0) ) \stackrel{\eqref{Eq:T0}}{\leq} \gamma\cdot \tilde \eta^{(t)}(y_0).
\end{align*}
Since $|\gamma| = 1$ and $|\nabla_{nn} u_t (y_0)| \leq C$, we arrive at \eqref{Eq:T1} after possibly slightly shrinking $c_0$.

Let us prove $\nabla_{nn} u_t(y_0) \leq C$. Since the matrix $\tilde W^t[u_t](y_0) = \{W^t_{\alpha\beta} [u_t] (y_0)\}_{1\leq \alpha, \beta \leq n-1}$ is diagonal with diagonal entries $\tilde \eta^{(t)} (y_0)$, we have
\begin{equation}
	\label{L1}
  \sum_{\alpha < n} \gamma_\alpha W^t_{\alpha \alpha}[u_t] (y_0) = \gamma \cdot \tilde\eta^{(t)} (y_0) = \zeta_t(y_0) + (1-t)\sum_{\alpha < n} \gamma_\alpha \nabla_{nn} u_t(y_0),
\end{equation}
and, by \eqref{Eq:T0},
\begin{equation}
	\label{L2}
\gamma \cdot \lambda' \geq d_{\Gamma'}(\lambda') 
, \quad
\forall \text{ ascending }\lambda' \in \Gamma'.
\end{equation}
Since $\underline{u}$ is a subsolution, we have $\lambda(\tilde W^t[\underline{u}] (y_0) ) \in \Gamma'$.
By \cite[Lemma 6.2]{CNS} and since $\gamma$ is descending, we have
\[
\sum_{\alpha < n} \gamma_\alpha  W^t_{\alpha \alpha} [\underline{u}] (y_0) \geq 
\sum_{\alpha < n} \gamma_\alpha \lambda_\alpha ( \tilde W^t[\underline{u}](y_0) ).
\]
Hence, by \eqref{L1} and \eqref{L2},
we obtain
\begin{multline}
\sum_{\alpha < n} \gamma_\alpha \Big(  W^t_{\alpha\alpha}[\underline{u}] - W^t_{\alpha\alpha}[u_t] \Big) (y_0)\\
\geq d_{\Gamma'} ( \lambda(\tilde W^t[\underline{u}] (y_0) )  ) - \zeta_t(y_0) - (1-t)\sum_{\alpha < n} \gamma_\alpha \nabla_{nn} u_t(y_0).
	\label{Eq:3.4+5}
\end{multline}

Now by the boundary condition $u_t - \underline{u} = 0$ on $\partial M$, we have
\[
\nabla_{\alpha \beta} (u_t-\underline{u}) (y) = - \nabla_n (u_t - \underline{u}) (y) b_{\alpha\beta}(y) \text{ for }  1 \leq \alpha, \beta \leq n-1,
\]
where $b_{\alpha \beta} = \langle \nabla_\alpha e_\beta, e_n \rangle$. Introducing $b_{\alpha\beta}^t = tb_{\alpha\beta} + (1-t) \sum_{\ell<n} b_{\ell\ell} \delta_{\alpha\beta}$, we then have
\begin{align}
\label{bb}
W^t_{\alpha\beta}[\underline{u}] - W^t_{\alpha\beta}[u_t] 
	&= \nabla_n (u_t - \underline{u}) \Big( b^t_{\alpha\beta} + \frac{n - 2 + (n-1)t}{2} \nabla_n (u_t + \underline{u}) \delta_{\alpha\beta}\Big)\\
		&\qquad	 - (1 - t) \nabla_{nn}(u_t - \underline{u})\delta_{\alpha\beta}\nonumber
\end{align}
on $\partial M$ near $y_0$. Defining
\begin{align*}
p^t
	&= \sum_{\alpha < n} \gamma_\alpha  b^t_{\alpha\alpha},\\
q^t
	&=\frac{n - 2 + (n-1)t}{2}  \sum_{\alpha < n} \gamma_\alpha,
\end{align*}
we deduce from \eqref{Eq:3.4+5} and \eqref{bb} that
\[
p^t\nabla_n (u_t - \underline{u}) + q^t \nabla_n(u_t - \underline{u})\nabla_n(u_t + \underline{u})\Big|_{y = y_0}
	 \geq d_{\Gamma'} ( \lambda(\tilde W^t[\underline{u}] (y_0) )  ) - \zeta_t(y_0) - C(1 - t).
\]
We may assume
$\zeta_t(y_0) \leq \frac{1}{4} d_{\Gamma'} ( \lambda(\tilde W^t[\underline{u}] (y_0) )  )$;
otherwise we are done. Then, for $t > 1 - \frac{1}{C}$,
\[
p^t\nabla_n (u_t - \underline{u}) + q^t \nabla_n(u_t - \underline{u})\nabla_n(u_t + \underline{u})\Big|_{y = y_0}
	\geq \frac{1}{2}d_{\Gamma'} ( \lambda(\tilde W^t[\underline{u}] (y_0) )  ) > 0.
\]
In particular, since $\nabla_n(u_t - \underline{u})(y_0) \geq 0$, there exists a positive constant $b_0 >0$
depending only on the bound for  $\|u\|_{C^1 (M)}$, $(f, \Gamma), (M, g)$ and $\underline{u}$ 
such that
\begin{equation}\label{b0b}
		p^t  + q^t  \nabla_n(u_t + \underline{u})\Big|_{y = y_0}
		\geq b_0.
\end{equation}

We now define the following function on $\Omega_{\delta, y_0}$:
\[\begin{aligned}
\Psi   = &p^t\nabla_n (u_t - \underline{u}) + q^t(|\nabla u_t|^2 - |\nabla \underline{u}|^2) 
	-  \sum_{\alpha < n} \gamma_\alpha  W^t_{\alpha\alpha}[\underline{u}]  
	+ (1 - t) \sum_{\alpha < n} \gamma_\alpha \nabla_{nn} \underline{u} + \zeta_t(y_0).
\end{aligned}\]
By \eqref{bb} and the fact that $y_0$ is a minimum point of $\zeta_t$ on $\partial M$, we see for $y \in \partial M \cap \partial \Omega_{\delta, y_0}$ that
\[
\Psi(y) \leq -\sum_{\alpha < n} \gamma_\alpha W_{\alpha\alpha}^t[u_t](y) + \gamma \cdot \tilde \eta^{(t)}(y).
\]
This implies on one hand that $\Psi(y_0) = 0$ and on the other hand, in view of \cite[Lemma 6.2]{CNS} and the order of $\gamma$ and $\tilde \eta^{(t)}$, that $\Psi \leq 0$ on $\partial M \cap \partial \Omega_{\delta, y_0}$. Applying $\mathcal{L}_t $ to $- \Psi$ and using Lemmas \ref{Lem:bbb} and \ref{lem:LtNu2}, we get
\[
- \mathcal{L}_t \Psi \leq  C \Big( 1 + \cT_t  + \sum_i \partial_i f_t(\lambda^{(t)}) |\lambda^{(t)}_i| \Big) - \frac{1}{C}  \sum_i \partial_i f_t(\lambda^{(t)}) (\lambda^{(t)}_i)^2.
\]
Using \eqref{f-lambda} with $\varepsilon$ small enough, we see that
\[
- \mathcal{L}_t \Psi \leq C \big( 1 + \cT_t\big).
\]

Now define $h = - \Psi + B \rho^2 + A v$, where $\rho, v$ are functions as before. By Lemma \ref{lemma}, we have
\[
\mathcal{L}_t h \leq 0 \; \mbox{in}\; \Omega_{\delta, y_0} \; \mbox{and} \; h \geq 0 \; \mbox{on} \; \partial \Omega_{\delta, y_0},
\]
when $A \gg B \gg 1$ depending on $\|u_t\|_{C^1(M)}$, $(f, \Gamma)$, $(M, g)$, $\underline{u}$ and $\psi$. By the maximum principle, we therefore obtain $\nabla_n h (y_0) \geq 0$, which implies that $\nabla_n \Psi (y_0) \leq C$. 

We proceed to bound $\nabla_{nn} u_t(y_0)$ from above. If $\nabla_{nn} (u_t + \underline{u} ) (y_0) < 0$, we are done. We hence assume $\nabla_{nn} (u_t + \underline{u} ) (y_0) \geq 0$. By Lemma \ref{lem:p1}, 
\[
\nabla_n \Psi(y_0)
	\geq p^t  \nabla_{nn} (u_t - \underline{u})    + q^t \nabla_n \big((\nabla_n u_t)^2 - (\nabla _n\underline{u})^2\big)\Big|_{y = y_0} - C.
\]
Since $\nabla_n (u_t - \underline{u}) \geq 0$ on $\partial M$, this implies that
\begin{align*}
\nabla_n \Psi(y_0)
	&\geq p^t  \nabla_{nn} (u_t - \underline{u})    + q^t \nabla_n(u_t + \underline{u}) \nabla_{nn}(u_t - \underline{u})\Big|_{y = y_0} - C\\
	&= \Big(p^t   + q^t \nabla_n(u_t + \underline{u}) \Big) \nabla_{nn}(u_t - \underline{u})\Big|_{y = y_0} - C.
\end{align*}
Combining this together with \eqref{b0b} and the fact that $\nabla_n\Psi(y_0) \leq C$, we conclude that $\nabla_{nn} u_t (y_0) \leq C$.
\end{proof}

\begin{lemma}
\label{lem:p2-t1fu}
If $\Gamma$ is of type 1 (in the sense of Definition \ref{defGType}) and $f$ is of unbounded type (in the sense of Definition \ref{def}), then Lemma \ref{lem:p2} holds.
\end{lemma}

\begin{proof} Fix a point $x_0 \in \partial M$ and choose a local orthonormal frame as before. Furthermore, we can assume under this frame that
\[
 W[u_t](x_0) =
\left( \begin{array}{cccc}
W_{11} & 0 & \cdots & W_{1n}\\
0 & W_{22} & \cdots & W_{2n}\\
\vdots & \vdots &  \ddots & \vdots\\
W_{n1} & W_{n2} & \cdots & W_{nn} \end{array} \right),
\]
i.e. $\{ W_{\alpha\beta}[u_t] (x_0) \}_{1\leq \alpha, \beta \leq n-1}$ is diagonal. Note that $\{ W^t_{\alpha\beta} (x_0) \}_{1\leq \alpha, \beta \leq n-1}$ is also diagonal. By Lemma \ref{lem:px}, we have that $|W^t_{ij}[u_t](x_0)| \leq C$ for $(i,j) \neq C$ and $|W^t_{nn}[u_t](x_0) - \nabla_{nn} u_t(x_0)| \leq C$. 
By \cite[Lemma 1.2]{CNS}, we know that if $\nabla_{nn} u_t(x_0)$
goes to infinity, then the eigenvalues $\eta^{(t)}$ of $W^t[u_t] (x_0)$ satisfy
\begin{equation}\label{eigen}
\left\{
\begin{aligned}
\eta^{(t)}_\alpha (x_0) = &\; W^t_{\alpha\alpha}[u_t] (x_0) + o(1), \; 1\leq \alpha \leq n-1,\\
\eta^{(t)}_n (x_0) = & \; W^t_{nn}[u_t] (x_0) \Big( 1 + O(\frac{1}{W^t_{nn}[u_t] (x_0)  }) \Big),
\end{aligned}
\right.
\end{equation}
where the implicit bound in the little $o$-term and big $O$-term depend only on $C$.

By Lemma \ref{key}, we have $d_{\Gamma'}(\tilde \eta^{(t)}(x_0) ) \geq c_0$. Therefore, when $\nabla_{nn} u_t (x_0)$ is large enough, by \eqref{eigen}, we have $d_{\Gamma'}((\eta^{(t)})' (x_0) ) > c_0/2$. Therefore, by Lemma \ref{lem:px}, $(\eta^{(t)})' (x_0)$ belongs to a compact subset $K'$ of $\Gamma'$ depending only $\Gamma$, on the bound in Lemma \ref{lem:px} and the constant $c_0$ above. By compactness of $K'$, there exists
$R_0 = R_0 (K',\Gamma) > 0$ such that $((\eta^{(t)})'(x_0), R_0) \in \Gamma$. By \eqref{Condition7} there is another constant $R_1 > 0$ depending only on $K', R_0, \psi$ and $\|u\|_{C^0(M)}$ such that
\[
f ((\eta^{(t)})'(x_0), R_0 + R_1) > \psi (x_0, u_t).
\] 
By \eqref{eqn} and \eqref{eigen}, this implies that $W^t_{nn}[u_t] (x_0) \leq C$ and hence $\nabla_{nn} u_t (x_0) \leq C$.
\end{proof}

It remains to consider the case $\Gamma$ is of type 1 and $f$ is of bounded type. In particular, the function $f_\infty$ in \eqref{f-infty} is a well-defined concave function in $\Gamma'$. Following \cite{Guan07, Guan14, Guan18, Trudinger}, we need to control $f_\infty ( \tilde\eta^{(t)}(x) ) - \psi (x, u_t)$ on $\partial M$. In this step, the bound for $\tilde\eta^{(t)}$ in Lemma \ref{key} is needed. We prove:

\begin{lemma}
\label{key2}
Under the hypotheses of Proposition \ref{prop}, suppose in addition that $\Gamma$ is of type 1 (in the sense of Definition \ref{defGType}) and $f$ is of bounded type (in the sense of Definition \ref{def}). There exists a constant $c_0 \in (0,1)$ depending on $\|u_t\|_{C^1(M)}$, $\|\nabla_{\alpha\beta} u_t\|_{C^0 (\partial M)}$, $\|\nabla_{\alpha n} u_t\|_{C^0 (\partial M)}$, $(f, \Gamma)$, $(M, g)$, $\underline{u}$, $\varphi$ and $\psi$,
such that
\[
m_t = \min_{x\in \partial M} \{f_\infty ( \tilde\eta^{(t)}(x) ) - \psi (x, u_t)\} \geq c_0 \text{ for } t \in [1 - c_0,1].
\]
\end{lemma}
\begin{proof}
Recall the notations $\mathcal{U}$, $F_\infty$ and $\mathcal{N}_\infty$ defined in Section \ref{Sec2}. By Lemmas \ref{lem:px} and \ref{key}, there exists compact subset $\mathcal{K}$ of $\mathcal{U}$ depending only on $\Gamma$ and the bounds in Lemmas \ref{lem:px} and \ref{key} such that $\tilde \eta^{(t)}(x) = \lambda(\tilde W^t[u_t](x)) \in \mathcal{K}$ for all $x \in \partial M$. For $B \in \mathcal{K}$ and $N \in \mathcal{N}_\infty(B)$, define
\[
\Upsilon_{t,B,N} := F_\infty(B) + \sum_{\alpha,\beta} N_{\alpha\beta}\cdot (\tilde W_{\alpha\beta}^t[u_t] - B_{\alpha\beta}) - (1-t)  \textrm{tr}(N)\nabla_{nn} u_t.
\]
We will show that there exists $c_0 \in (0,1)$ depending on 
$\|u_t\|_{C^1(M)}$, $\|\nabla_{\alpha\beta} u_t\|_{C^0 (\partial M)}$, $\|\nabla_{\alpha n} u_t\|_{C^0 (\partial M)}$, $(f, \Gamma)$, $(M, g)$, $\underline{u}$, $\varphi$ and $\psi$,
such that for all $B \in \mathcal{K}$ and $N \in \mathcal{N}_\infty(B)$,
\begin{equation}
\Upsilon_{t,B,N}(x) -\psi (x, u_t) \geq 2c_0 \; \mbox{on}\; \partial M \text{ for } t \in [1-c_0,1].
	\label{Eq:T4}
\end{equation}
Once this is done, since $\nabla_{nn} u_t \geq - C$ (by Corollary \ref{Cor:X}) and $0 \leq \mathrm{tr}(N) \leq C$ (by Lemma \ref{Lem:Finfty}(iv)), it follows that, for all $B \in \mathcal{K}$ and $N \in \mathcal{N}_\infty(B)$,
\begin{multline*}
F_\infty(B) + \sum_{\alpha,\beta} N_{\alpha\beta}\cdot (\tilde W_{\alpha\beta}^t[u_t] - B_{\alpha\beta}) - \psi (x, u_t)\\
	\geq 2c_0 + (1-t)  \textrm{tr}(N)\nabla_{nn} u_t \geq c_0 \; \mbox{on}\; \partial M \text{ for } t \in [1-c_0,1], 
\end{multline*}
which then implies the conclusion in view of \eqref{Eq:T3}.

Fix $B \in \mathcal{K}$ and $N \in \mathcal{N}_\infty(B)$. To lighten up the notation, we will abbreviate $\Upsilon_{t,B,N}$ to $\Upsilon_t$. Suppose that $\Upsilon_t - \psi(\cdot,u_t)$ attains its minimum on $\partial M$ at $y_0 \in \partial M$.
Choose a smooth local orthonormal frame
$e_1, \ldots, e_n$ around $y_0$ as before.
Our aim is to prove that $\nabla_{nn} u_t (y_0) \leq C$. Once this is achieved, the conclusion is obtained as follows: By Lemmas \ref{lem:px} and and \ref{key}, $\tilde \eta^{(t)}(\partial M)$ is contained in a compact set $K'$ of $\Gamma'$ (which depends only on $\Gamma$ and the constants in Lemmas \ref{lem:px} and \ref{key}). We may thus pick some $c_0 > 0$ depending only on $f$, $K'$, $\Gamma$ such that 
\[
\min_{\lambda' \in K'} \{ f_\infty(\lambda') - f(\lambda',C)\} \geq 3c_0 > 0.
\]
By the monotonicity property \eqref{f2} of $f$, we then have
\begin{multline*}
3c_0 
	\leq f_\infty (  \tilde\eta^{(t)}(y_0)) - f(\tilde\eta^{(t)}(y_0),C)\\
	 \stackrel{\eqref{f2}}{\leq} f_\infty (  \tilde\eta^{(t)}(y_0)) - f(\eta^{(t)}(y_0))
 	= F_\infty (\tilde W^t[u_t](y_0)) - \psi(y_0,u_t),
\end{multline*}
which in view of \eqref{Eq:T3} implies
\[
3c_0 
	\leq F_\infty(B) + \sum_{\alpha,\beta} N_{\alpha\beta}\cdot (\tilde W_{\alpha\beta}^t[u_t] (y_0) - B_{\alpha\beta}) - \psi(y_0,u_t) .
\]
Since $\nabla_{nn} u_t(y_0) \leq C$ and $\|N\|\leq C$ (see Lemma \ref{Lem:Finfty}(iv)), we deduce \eqref{Eq:T4} after possibly shrinking $c_0$.

We turn to prove $\nabla_{nn} u_t (y_0) \leq C$. By \eqref{bb} and \eqref{F}  as well as the fact that $\psi_z \leq 0$ and $u_t \geq \underline{u}$, we have
\begin{align*}
 & \nabla_n (u_t - \underline{u}) \sum_{\alpha,\beta} N_{\alpha\beta} \Big(  b^t_{\alpha\beta}
 +  \frac{n-2 + (n-1)t}{2}  \nabla_n (u_t + \underline{u}) \delta_{\alpha\beta} \Big) (y_0)\\
 	&\qquad = \sum_{\alpha,\beta}N_{\alpha\beta} \big( W^t_{\alpha\beta} [\underline{u}] (y_0) - W^t_{\alpha\beta}[u_t] (y_0) \big) + (1-t) \mathrm{tr}(N) \nabla_{nn} (u_t - \underline{u})(y_0)\\
	&\qquad = F_\infty(B) + \sum_{\alpha,\beta}N_{\alpha\beta} \big( W^t_{\alpha\beta} [\underline{u}] (y_0) - B \big)  - \Upsilon_t(y_0) - (1-t) \mathrm{tr}(N) \nabla_{nn}  \underline{u}(y_0)\\
	&\qquad \geq F_{\infty} (\tilde W^t[\underline{u}] (y_0) )- \Upsilon_t(y_0) - C(1 - t)
		 =  f_\infty(\lambda(\tilde W^t[\underline{u}]) (y_0)) - \Upsilon_t(y_0) - C(1 - t)\\
	&\qquad \geq [f_\infty(\lambda(\tilde W^t[\underline{u}]) (y_0)) - \psi(y_0,\underline{u})]- [\Upsilon_t(y_0) - \psi(y_0,u_t)] - C(1-t).
\end{align*}
Since $\underline{u}$ is a subsolution, we have that $f_\infty(\lambda(\tilde W^t[\underline{u}])) - \psi(\cdot,\underline{u})$ is positive on $\partial M$ and hence bounded from below by a positive constant, say $\underline{m} > 0$, which depends only on $(f,\Gamma)$, $(M,g)$, $\psi$ and $\underline{u}$. To proceed, note that we may assume that $\Upsilon_t(y_0) - \psi(y_0,u_t) \leq \underline{m}/4$, as otherwise we are done. Then, for $t > 1 - \frac{1}{C}$,
\[
\nabla_n (u_t - \underline{u})N_{\alpha\beta} \Big(  b^t_{\alpha\beta}
 +  \frac{n-2 + (n-1)t}{2}  \nabla_n (u_t + \underline{u}) \delta_{\alpha\beta} \Big) (y_0)
 	\geq \frac{1}{2}\underline{m}.
\]
As $\nabla_n (u_t - \underline{u}) \geq 0$, this implies that
\begin{equation}
\label{bb2x}
\sum_{\alpha,\beta}N_{\alpha\beta} \Big(  b^t_{\alpha\beta}
 +  \frac{n-2 + (n-1)t}{2}  \nabla_n (u_t + \underline{u}) \delta_{\alpha\beta} \Big) (y_0)
 	\geq \frac{1}{C}.
\end{equation}

We define the following function in $\Omega_{\delta, y_0}$:
\begin{align*}
\Phi (y) &= p_*^t\nabla_n ( u_t-  \underline{u} ) 
	+ q_*^t\big( |\nabla u_t|^2 - | \nabla \underline{u}|^2 \big)\\
	&\qquad - F(B)
		- \sum_{\alpha,\beta} N_{\alpha \beta} \big(W^t_{\alpha\beta} [\underline{u} ] (y) - B_{\alpha\beta} \big)\\
	&\qquad
+ (1-t) \mathrm{tr}(N) \nabla_{nn} \underline{u}
+ \psi (y,u_t) + \Upsilon_t(y_0) - \psi (y_0, u_t),
\end{align*}
where
\begin{align*}
p_*^t &= \sum_{\alpha,\beta} N_{\alpha\beta} 
  b^t_{\alpha \beta},\\
q_*^t &=  \frac{n-2 + (n-1)t}{2} \mathrm{tr}(N).
\end{align*}
Note that, along $\partial M \cap \partial \Omega_{\delta, y_0}$, we have by \eqref{bb} that
\begin{align*}
\Phi (y) &= - [\Upsilon_t(y) - \psi (y,u_t)] + [\Upsilon_t(y_0) - \psi (y_0, u_t)].
\end{align*}
This shows that $\Phi \leq 0 $ along $\partial M \cap \partial \Omega_{\delta, y_0}$ and $\Phi(y_0) = 0$. Moreover, by applying $\mathcal{L}_t $ to $- \Phi$ and using Lemmas \ref{Lem:bbb} and \ref{lem:LtNu2} as in the proof of Lemma \ref{key} (keeping in mind the positive semi-definiteness of $N$ from Lemma \ref{Lem:Finfty}(i)), we compute
\begin{align*}
- \mathcal{L}_t \Phi 
	&\leq - \frac{n-2 + (n-1)t }{2} \mathrm{tr}(N) \sum_i \partial_i f_t(\lambda^{(t)}) (\lambda^{(t)}_i)^2\\
		&\qquad + C\mathrm{tr}(N) \big(  1 + \cT_t  + \sum_i \partial_i f_t(\lambda^{(t)}) |\lambda^{(t)}_i| \big) 
			 + C \big( 1 + \cT_t \big).
\end{align*}
Therefore, by \eqref{f-lambda} and Lemma \ref{Lem:Finfty}(iv),
\[
-\mathcal{L}_t\Phi \leq C \big( 1 + \cT_t \big).
\]

Now define $h = - \Phi + B\rho^2 + Av$ on $\Omega_{\delta, y_0}$.
Choosing $A\gg B \gg 1$ which depend on $\|u_t\|_{C^1(M)}$, $C(\mathcal{K})$, $(f, \Gamma)$, $(M, g)$, $\underline{u}$ as well as  $\psi$, and according to Lemma \ref{lemma}, we obtain
\[
\mathcal{L}_t h \leq 0 \; \mbox{in} \; \Omega_{\delta, y_0}\; \mbox{and} \; h\geq 0 \;\mbox{on} \; \partial \Omega_{\delta, y_0}.
\]
By maximum principle, we derive that $h \geq 0$ in $\Omega_{\delta, y_0}$ and $\nabla_n h (y_0) \geq 0$. Therefore, $\nabla_n \Phi (y_0) \leq C$.

We can now proceed to bound $\nabla_{nn} u_t(y_0)$ from above as in the proof of Lemma \ref{key}. If $\nabla_{nn} (u_t + \underline{u} ) (y_0) < 0$, we are done. We hence assume $\nabla_{nn} (u_t + \underline{u} ) (y_0) \geq 0$. By Lemma \ref{lem:p1}, 
\[
\nabla_n \Phi(y_0)
	\geq p_*^t  \nabla_{nn} (u_t - \underline{u})    + q_*^t \nabla_n \big((\nabla_n u_t)^2 - (\nabla _n\underline{u})^2\big)\Big|_{y = y_0} - C.
\]
Since $\nabla_n (u_t - \underline{u}) \geq 0$ on $\partial M$, this implies that
\begin{align*}
\nabla_n \Phi(y_0)
	&\geq p_*^t  \nabla_{nn} (u_t - \underline{u})    + q_*^t \nabla_n(u_t + \underline{u}) \nabla_{nn}(u_t - \underline{u})\Big|_{y = y_0} - C\\
	&= \Big(p_*^t   + q_*^t \nabla_n(u_t + \underline{u}) \Big) \nabla_{nn}(u_t - \underline{u})\Big|_{y = y_0} - C.
\end{align*}
Combining this with \eqref{bb2x} and the fact that $\nabla_n\Phi(y_0) \leq C$, we conclude that $\nabla_{nn} u_t (y_0) \leq C$.
\end{proof}

\begin{lemma}
\label{lem:p2-t1fb}
If $\Gamma$ is of type 1 (in the sense of Definition \ref{defGType}) and $f$ is of bounded type (in the sense of Definition \ref{def}), then Lemma \ref{lem:p2} holds.
\end{lemma}

\begin{proof} We fix $x_0 \in \partial M$ and set up as in the proof of Lemma \ref{lem:p2-t1fu}. We knew that, when $\nabla_{nn} u_t(x_0)$ is sufficiently large,
$\tilde\eta^{(t)}(x_0)$ and $(\eta^{(t)})' (x_0)$ belong to a compact subset $K'$ of $\Gamma'$.

By Lemma \ref{key2}
\begin{equation}
\label{low}
f_\infty (\tilde \eta^{(t)} )  - \psi (x , u_t) \geq c_0 > 0 \text{ on } \partial M.
\end{equation}
Hence, there exists $R_1 = R_1(f,K',\Gamma)$ and $\delta_1 = \delta_1(f,K',\Gamma) > 0$ such that
\[
f(\tilde \eta^{(t)}, R_1) - \psi (x , u_t) \geq \frac{1}{2}c_0. 
\]
and
\[
f(\lambda', R_1) - \psi (x , u_t) \geq \frac{1}{4}c_0 \text{ for all } |\lambda' - \tilde\eta^{(t)}| \leq \delta_1.
\]

Now, by \eqref{eigen}, if $\nabla_{nn} u_t(x_0)$ is too large, we then have $|\tilde\eta^{(t)} - (\eta^{(t)})'|(x_0) \leq \delta_1$ and $\eta^{(t)}_n(x_0) \geq R_1$ which then leads to
\[
f_t(W[u_t](x_0)) - \psi (x , u_t) = f(\eta^{(t)}(x_0)) - \psi (x , u_t) \geq \frac{1}{4}c_0 > 0,
\]
which is a contradiction. Thus $\nabla_{nn} u_t(x_0) \leq C$.
\end{proof}

\begin{proof}[Proof of Lemma \ref{lem:p2}]
The result is a combination of Lemmas \ref{lem:p2-t2}, \ref{lem:p2-t1fu} and \ref{lem:p2-t1fb}.
\end{proof}

\section{Existence of non-smooth solutions}\label{Sec:ENS}

\begin{proof}[Proof of Example \ref{Ex1}]
	
	The Schouten tensor of $g$ is 
	\begin{align*}
		A_g 
		&=  -  \frac{1}{2 }   dt^2    +  \frac{1}{2 } h.
	\end{align*} 
	We look for a solution to \eqref{Eq:1.3Cyl} of the form $u = u(t)$. We have
	\begin{align*}
		W[u]
		&= \big(\ddot u + \frac{1}{2}(1 -  \dot u^2)\big) dt^2	 
		+  \frac{1}{2} (1 - \dot u^2)  h,
	\end{align*}
	where a dot is used to denote differentiation with respect to $t$. Problem \ref{Eq:1.3Cyl} thus becomes 
	\begin{equation}
		\begin{cases}
			(1 - \dot u^2)^{k-1}\Big(\ddot u + \frac{n - 2k}{2k}(1 - \dot u^2)\Big) = \frac{n}{2k} e^{-2ku} &\text{ in } (-\ell, \ell),\\
			1 - \dot u^2 > 0 &\text{ in } (-\ell, \ell),\\
			u(\pm \ell) = c.
		\end{cases}
		\label{Eq:OBVP}
	\end{equation}

	The ODE on the first line of \eqref{Eq:OBVP} has a first integral: If we define $H$ by
	\[
	H(x, y) = e^{(2k-n)x} (1 - y^2)^k -  e^{-nx},
	\]
	then $H(u, \dot u)$ is constant along a solution. 
	
	For any $d \in (-\infty, 0)$, let $u_d$ denote the unique classical solution to the initial value problem
	\begin{equation}
		\begin{cases}
			(1 - \dot u_d^2)^{k-1}\Big(\ddot u_d + \frac{n - 2k}{2k}(1 - \dot u_d^2)\Big) = \frac{n}{2k} e^{-2ku}  ,\\
			u_d(0) = d, \dot u_d(0) = 0
		\end{cases}
		\label{Eq:OIVP}
	\end{equation}
	in its maximal interval of existence $(-T_d, T_d)$. It is routine to show that $1 - \dot u_d^2 > 0$ in $(-T_d, T_d)$, $H(u_d, \dot u_d) = H(d,0) < 0$, $T_d$ is finite and given by
	\[
	T_d = \int_d^{-\frac{1}{n} \ln |H(d,0)|} \Big[1 - e^{\frac{n-2k}{k}x} \Big( e^{-nx} +  H(d,0)\Big)^{\frac{1}{k}}\Big]^{-\frac{1}{2}}\,dx < 0,
	\]
	and 
	\[
	\begin{cases}
		u_d(t) \rightarrow -\frac{1}{n} \ln |H(d,0)|,\\
		\dot u_d(t) \rightarrow \pm 1,\\
		\ddot u_d(t) \rightarrow \infty
	\end{cases}
	\text{ as }t \rightarrow \pm T_d.
	\]
	
	Now, for any given $c \in \RR$, since the function $x \mapsto H(x,0)$ increases from $-\infty$ to $0$ as $x$ increases from $-\infty$ to $0$, we can find a unique $d_c$ such that $-\frac{1}{n} \ln |H(d_c,0)| = c$. The conclusion follows with $\ell = T_{d_c}$ and $u = u_{d_c}$.
\end{proof}

\end{document}